\documentclass{article}
\usepackage{amsmath,amssymb,epsfig,subfigure}
\usepackage{exscale}
\newtheorem{theorem}{Theorem}

\newtheorem{remark}[theorem]{{\it Remark\/}}

\newtheorem{proposition}[theorem]{{\it Proposition\/}}
\newtheorem{definition}[theorem]{{\it Definition\/}}

\newenvironment{proof}{{\noindent \bf Proof:}}{\hfill $\fbox{}$ \vspace*{5mm}}

\newcommand{\ch}{{\cal H}}
\newcommand{\cl}{{\cal L}}

\newcommand{\diag}{{\rm diag}}

\newcommand{\imm}{\mathrm{i}}

\DeclareMathOperator*{\argmin}{argmin}

\begin{document}
\title{Fast Preconditioners for Total Variation Deblurring with Anti-Reflective Boundary Conditions}

\author{Zheng-Jian Bai\thanks{School of Mathematical Sciences, Xiamen
University, Xiamen 361005, People's Republic of China, Dipartimento
di Fisica e Matematica, Universit{\`{a}} dell'Insubria - Sede di
Como, Via Valleggio 11, 22100 Como, Italy, \textbf{E-mail:}
zjbai@xmu.edu.cn. The research of this author was partially supported by the Natural Science Foundation of Fujian Province of China for Distinguished Young Scholars (No. 2010J06002), NCETXMU, and SRF for ROCS, SEM, and
Internationalization Grant of U. Insubria 2008, 2009.} \and Marco Donatelli
\thanks{Dipartimento di Fisica e Matematica,
Universit{\`{a}} dell'Insubria - Sede di Como, Via Valleggio 11,
22100 Como, Italy, \textbf{E-mail:}
$\{$marco.donatelli,stefano.serrac$\}$@uninsubria.it.
The work of these authors was partially supported by MIUR, grant number 20083KLJEZ
and 2006017542.
}
\and Stefano Serra-Capizzano$\mbox{ }^{\dagger}$
}

\maketitle
\begin{abstract}
In recent works several authors have proposed the use of
precise boundary conditions (BCs) for blurring models and they
proved that the resulting choice (Neumann or reflective,
anti-reflective) leads to fast algorithms both for deblurring and
for detecting the regularization parameters in presence of noise.
When considering a symmetric point spread function, the
crucial fact is that such BCs are related to fast trigonometric
transforms.

In this paper we combine the use of precise BCs with the Total
Variation (TV) approach in order to preserve the jumps of the given
signal (edges of the given image) as much as possible. We consider a
classic fixed point method with a preconditioned Krylov method
(usually the conjugate gradient method) for the inner iteration.
Based on fast trigonometric transforms, we propose some preconditioning strategies
which are suitable for reflective and anti-reflective BCs.
A theoretical analysis motivates the choice of our preconditioners and
an extensive numerical experimentation is reported and critically discussed.
The latter shows that the TV regularization with anti-reflective BCs
implies not only a reduced analytical error, but also
a lower computational cost of the whole restoration procedure
over the other BCs.
\end{abstract}
{\bf AMS-SC}: 65F10, 65F15, 65Y20.

\section{Introduction}\label{intro}

We are concerned with specific linear algebra/matrix theory
aspects of the vast field of inverse problems \cite{hanke,G93} which
model the blurring of signals and images ($2D$ or $dD$ with $d\ge
3$). Here the goal is to reconstruct the real object from its
blurred and noisy version and this goal is a classical one in
astronomical imaging, medical imaging, geosciences, etc.
\cite{bertero}.

The blurring model is assumed to be space-invariant, i.e., the point
spread function (PSF) is represented by a specific bivariate
function $h(x-y)$ ($x, y\in\Omega$) for some univariate function
$h(\cdot)$ \cite{nagybook}. According to the linear models
described in the literature \cite{G93}, the observed signal or image
$v$ and the original signal or image $u$ are described by the
relation
\begin{equation} \label {nbf}
v(x) = \ch u(x) + \eta(x):= \int_\Omega h(x-s)u(s)ds + \eta(x),
\quad x\in\Omega,
\end{equation}
where the kernel $h$ is the PSF and $\eta$ denotes the noise. The
problem (\ref{nbf}) is ill-posed since the operator $\ch$ is
compact \cite{G93}. Therefore, the approximation/discretization
matrix of $\ch$ is usually increasingly ill-conditioned when the
number $n$ of pixels becomes large. In addition, the size of the
subspace associated with small eigenvalues, which substantially
intersects the high frequencies, is large and proportional to the
size of the matrix. Thus, we cannot directly solve $\ch u = d$,
since the small perturbations, represented by the noise $\eta$ with
important high frequency components due to its probabilistic nature,
would be amplified unacceptably.

To remedy to the latter essential ill-conditioning of problem
(\ref{nbf}), one may employ regularization methods. The Total
Variation (TV) regularization approach is a good choice for
restoring edges of the original signals \cite{TV}.
Rudin, Osher, and Fatemi \cite{TV} gave the total variation functional
in the form
\begin{equation} \label {tv}
J_{TV}(u) := \int_\Omega \left|\nabla u\right|dx,
\end{equation}
where $|\cdot|$ denotes the Euclidean norm. We note that the
Euclidean norm $|\cdot|$ is not differentiable at zero. To avoid the
non-differentiability, Acar and Vogel \cite{AV94} considered the
following minimization
\begin{equation} \label {mp}
\min_u \left\{ \|\ch u-v\|_{L^2(\Omega)} + \alpha \int_\Omega
\sqrt{|\nabla u|^2+\beta^2}\, dx \right\},
\end{equation}
where $\alpha,\beta$ are positive parameters. Notice that the
penalty term
$\int_\Omega \sqrt{|\nabla u|^2+\beta^2}\, dx$
converges to $J_{TV}(u)$ as $\beta\to 0$. In
other words, the latter is a differentiable regularized version of
$J_{TV}(u)$. The corresponding Euler-Lagrange equation for
(\ref{mp}) is given by
\begin{equation} \label{el}
\left\{
\begin{array}{rl}
   g(u) := \ch^*(\ch u - v) -\alpha \,  \cl_u(v) = 0, &  x\in\Omega,\\[2mm]
  \frac{\partial u}{\partial n} = 0, & x\in\partial\Omega,
\end{array}
\right.
\end{equation}
where $*$ denotes the adjoint operator and
$
\cl_u(y):= - \nabla\cdot\left(\frac{1}{\sqrt{|\nabla
u|^2+\beta^2}}\nabla y\right)
$
is the differential operator appearing in (\ref{el}) and comes from
the regularized penalty term. Vogel and Oman \cite{VO96} proposed a
lagged diffusivity fixed point (FP) iteration for solving
(\ref{el}). More precisely, given the initial guess $u^0$, the new
iterate $u^{k+1}$ is obtained by $u^{k}$ thanks to the equation
\begin{equation} \label{fp}
A_{u^k}u^{k+1} \equiv\left(H^*H  + \alpha\, L(u^k)\right)u^{k+1} =
H^*v,\quad k =0,1,\ldots,
\end{equation}
where $H$ and $L(u^k)u^{k+1}$ denote the
discretization/approximation matrix of $\ch$ and
$\cl_{u^k}(u^{k+1})$, respectively. We use the compound mid point rule
and standard centered finite differences of precision order two for the finite dimensional approximation of
$\ch$ and $\cl_u(\cdot)$, respectively. Therefore, at each FP iteration,
we may use the preconditioned conjugate gradient (PCG) method \cite[Algorithm 10.3.1]{GV96} for
solving the linear system \eqref{fp}.

In \cite{CCW99} the authors proposed a cosine preconditioner when
$H$ is a Toeplitz matrix, i.e., in the case of zero-Dirichlet boundary conditions (BCs).
However, the choice of such BCs induces remarkable
pathologies in the quality of the restored images, which should be
avoided or at least minimized. In reality, using classical BCs such as periodic or zero-Dirichlet may imply, when the
background is not uniformly black, disturbing Gibbs phenomena called
ringing effects \cite{NCT,S03,nagybook}.

The novelty of this paper is represented by the choice of
appropriate BCs, in order to reduce the ringing effects, and in the related
matrix/numerical analysis. The latter
will affect the algebraic expression of $H$, while for $L(u^k)$ the
choice of the BCs in the Euler-Lagrange equation for (\ref{el})
seems to impose Neumann BCs. In this context, the idea is to combine
the application of anti-reflective BCs, already
studied for their precision with plain regularization methods like
Tikhonov and Landweber \cite{S03,perrone,Shi,DES06,ADNS09,ChHa08},
with the more sophisticate TV regularization. In other
words, the first aim consists in checking how to reduce the ringing effects and the over-smoothing of the
edges simultaneously. Next, we want to study the use of preconditioners based
on innovative fast transforms in the setting of Krylov methods when
the real problem is modeled by a symmetric PSF. The final goal is to
combine the precision of the reconstruction with highly efficient
numerical procedures. We study some preconditioning techniques and give the theoretical
explanations of different proposals. An effective preconditioner for the reflective
BCs is inspired by the work in \cite{CCW99}, while, for the
anti-reflective BCs, we propose a new sine preconditioner for the
linear system \eqref{fp} and explore the re-blurring approach
introduced in \cite{DS05}.
Numerical results confirm the effectiveness of the proposed preconditioners
and the superiority of the antireflective BCs with respect to the reflective BCs.
Indeed, antireflective BCs not only provide better restorations as expected
(see \cite{S03,DES06}), but also require a lower computational cost. The latter
is a consequence of the fact that a more precise model requires lesser
regularization, i.e. a smaller $\alpha$, and the convergence of our approach
is fast for small $\alpha$.

The paper is organized into seven more sections. In Section
\ref{BCs} we consider reflective and anti-reflective BCs. Sections
\ref{prec-refl} is devoted to define
optimal preconditioners for signal and image deblurring for the
reflective BCs and anti-reflective BCs.
In Section \ref{sect:spectr} some spectral features of the proposed
preconditioners are discussed. Section \ref{num2} is
concerned with the numerical tests for checking the real efficiency
of the considered preconditioners and the quality of the
restorations. Finally, in Section \ref{final} we draw conclusions.

\section{Boundary Conditions}\label{BCs}
We start by introducing the one-dimensional deblurring problem.
Consider the original signal $ \tilde u = (\ldots,
u_{-m+1},\ldots,u_0, u_1,\ldots, u_n,$ $u_{n+1}, \ldots, u_{n+m},
\ldots)^T$ and the normalized blurring PSF given \nolinebreak by

\begin{equation}\label{blur1d}
h = (\ldots, 0, 0, h_{-m} , h_{-m+1},\ldots , h_0 , \ldots ,
h_{m-1},h_m, 0,0,\ldots)^T,
\end{equation}
with the usual normalization, i.e., $\sum_{j=-m}^m h_j=1$ which
preserves the global intensity and therefore represents an average.
The blurred signal $v$ is the convolution of $h$ and $\tilde u$ and
consequently $v_i =\sum_{j=-\infty}^{\infty}h_j u_{i-j}$ is such that
\begin{equation}\label{blurred1d:matrixform}
v=\left(\begin{array}{cccccccccc}
h_m & \cdots & h_0 & \cdots & h_{-m}  \\
    & h_m    &     & h_0    &       & h_{-m} &   &   &  0  \\
    &        & \ddots & \ddots & \ddots & \ddots & \ddots \\
    &        &        & \ddots & \ddots & \ddots & \ddots & \ddots \\
    & 0      &        &        & h_m    &     & h_0    &       & h_{-m} \\
    &        &        &        &     & h_m & \cdots & h_0 & \cdots & h_{-m}
\end{array}
\right)
\left(\begin{array}{c}
u_{-m+1} \\
u_{-m+2} \\
\vdots  \\
u_0   \\
u   \\
u_{n+1} \\
\vdots \\
u_{n+m-1} \\
u_{n+m}
\end{array}
\right).
\end{equation}
The deblurring problem is to recover the vector $u=(u_1,\ldots,u_n)^T$
given the blurring function $h$ and a blurred signal
$v=(v_1,\ldots,v_n)^T$ of finite length,
Thus the blurred signal $v$ is determined not only by $u$, but also
by $(u_{-m+1},\ldots,u_0)^T$ and $(u_{n+1},\ldots,u_{n+m})^T$ and
the linear system (\ref{blurred1d:matrixform}) is underdetermined.
To overcome this, we make certain assumptions, that is the BCs on the unknown boundary data $u_{-m+1},\ldots,u_0$ and
$u_{n+1},\ldots,u_{n+m}$ in such a way that the number of unknowns
equals the number of equations.

For the \emph{zero-Dirichlet} BCs, we assume that
the data outside $u$ are zero, i.e., we set $u_{1-j}=u_{n+j}=0$ for
$j=1,\ldots,m$. Then, (\ref{blurred1d:matrixform})
becomes $Au = v$, where $A$ is Toeplitz.
For the \emph{periodic} BCs, we assume that the
signal $u$ is extended by periodicity. More precisely,  we set $u_{1-j}=u_{n-j+1}$
and $u_{n+j}=u_{j}$ for $j=1,\ldots,m$. It follows that (\ref{blurred1d:matrixform})
becomes $Au = v$, where $A$ is circulant and hence it can be diagonalized
by the Discrete Fourier Transform (DFT) (see \cite{nagybook}).

For the \emph{Neumann or reflective} BCs, we assume that
the data outside $u$ are a reflection of the data inside $u$ (refer
to \cite{NCT}). More precisely, we set $u_{1-j}=u_j$
and $u_{n+j}=u_{n+1-j}$ for all $j=1,\ldots,m$ in
(\ref{blurred1d:matrixform}). Thus (\ref{blurred1d:matrixform})
becomes $Au = v$, where $A$ is neither Toeplitz nor circulant but a
special $n$-by-$n$ Toeplitz plus Hankel matrix which is diagonalized
by the discrete cosine transform provided that the blurring function
$h$ is symmetric, i.e., $h_j = h_{-j}$ for all $j$ in
(\ref{blur1d}). It follows that the above system can be solved by
using three fast cosine transforms (FCTs) in $O(n \log n)$
operations \cite{NCT}. This approach is computationally
interesting since the FCT requires only real operations and is about
twice as fast as the FFT and this is true in two dimensions as well.
We note that the reflection ensures the continuity of the signal and
the error is linear in the discretization step (the latter can be
easily seen by applying a Taylor expansion \cite{NCT}). Therefore,
we usually observe a reduction of the boundary artifacts with
respect to zero-Dirichlet and periodic BCs.

For the \emph{anti-reflective} BCs, we assume that the data
outside $u$ are an anti-reflection of the data inside $u$. More
precisely, if $x$ is a point outside the domain and $x^*$ is the
closest boundary point, then we have $x=x^*-\delta x$ and the
quantity $u(x)$ is approximated by $u(x^*)-(u(x^*+\delta
x)-u(x^*))$. Consequently, we set
\begin{equation}\label{anti1d}
    \begin{array}{cc}
u_{1-j}=u_1-(u_{j+1}-u_1)=2u_1-u_{j+1}, & \ \ \ {\rm for\ all}\
j=1,\ldots,m,
\\
u_{n+j}=u_n-(u_{n-j}-u_n)=2u_n-u_{n-j}, & \ \ \ {\rm for\ all}\
j=1,\ldots,m
\end{array}
\end{equation}
in (\ref{blurred1d:matrixform}). By a Taylor expansion, the
anti-reflection \eqref{anti1d} ensures a $C^1$ continuity of the
signal and the error is quadratic in the discretization step
\cite{S03}. Usually, the boundary artifacts are reduced
also with respect to reflective BCs.

Imposing the anti-reflection \eqref{anti1d}, the linear system
\eqref{blurred1d:matrixform} becomes $Au=v$, where $A$ is a Toeplitz
plus Hankel plus a rank-2 correction matrix, where the correction is
placed at the first and the last column. Furthermore, in
\cite{ADNS09} the authors proved that if $h$ is symmetric then
$A=T_n \Lambda T_n^{-1}$ where
\[
T_n = \left[
        \begin{array}{ccc}
          1 & 0 & 0 \\
          p & S_{n-2} & Jp \\
          0 & 0 & 1 \\
        \end{array}
      \right],\quad T_n^{-1} = \left[
        \begin{array}{ccc}
          1 & 0 & 0 \\
         -S_{n-2}p & S_{n-2} & -S_{n-2}Jp \\
          0 & 0 & 1 \\
        \end{array}
      \right],
\]
where $J$ is the flip matrix, $S_{n-2}$ is the sine transform matrix of order $n-2$, and where
$p_j=1-j/(n-1)$ so that the first column vector is exactly the sampling of
the function $1-x$ on the grid $j/(n-1)$ for $j=0,\ldots,n-1$.
Finally, $\Lambda$ is a diagonal matrix given by suitable samplings of the
function
\begin{equation}\label{eq:symb}
    \hat h(y)=\sum h_j {\rm exp}(\imm jy),
\end{equation}
which is the symbol generated by the PSF. That is,
\begin{equation}\label{eq:lambda}
    \Lambda = \diag_{y=1,\dots,n}\left(\hat{h}(y_j)\right), \quad
    \quad S_m = \sqrt{\frac{2}{m+1}}\left(
\sin\left(\frac{ji\pi}{m+1}\right)\right)_{i,j=1}^{m},
\end{equation}
where
\begin{equation}\label{eq:gridy}
    y_j=\frac{(j-1)\pi}{n-1},  \mbox{ for } j=1,\dots,n-1, \qquad \mbox{and} \qquad y_n=0.
\end{equation}
As a consequence, a generic system $A u=v$ can be solved within $O(n
\log n)$ real operations by resorting to the application of three
fast sine transforms (FSTs) (refer to \cite{S03}), where each FST is
computationally as cheap as a generic FCT.
%
There is a suggestive functional interpretation of the transform
$T_n$. The transform associated with periodic BCs matrices is the Fourier transform: its $j$-th column vector, up
to a normalizing scalar factor, can be viewed as a sampling, over a
suitable uniform gridding of $[0,2\pi]$, of the frequency function
${\rm exp}(-{\bf i} jy)$. Analogously, when imposing reflective BCs
with a strongly symmetric PSF, the transform of the related reflective BCs
matrices is the cosine transform: its $j$-th column vector, up to a
normalizing scalar factor, can be viewed as a sampling, over a
suitable uniform gridding of $[0,\pi]$, of the frequency function
$\cos(jy)$. Here the imposition of the anti-reflective BCs, by
operating a central symmetry with respect to any point of frontier,
can be functionally interpreted as a linear combination of sine
functions and of linear polynomials (whose use is exactly required
for imposing the $C^1$ continuity at the borders). This intuition
becomes evident in the expression of $T_n$. Indeed
\begin{equation}\label{eq:Tn}
    T_n=\left(\,1-\frac{y}{\pi}\,,\,\sin(y)\,,\,\cdots\,,\,
\sin((n-2)y)\,,\,\frac{y}{\pi}\,\right)\ \cdot\
{\rm diag}\left(1\,,\,\sqrt{\frac{2}{n-1}}\,I_{n-2}\,,\,1\right),
\end{equation}
where $y$ is defined \eqref{eq:gridy} and it is a suitable gridding of $[0, \pi]$.

Now, we introduce the two-dimensional case. For the Neumann or
reflective BCs, the blurring matrix is a block Toeplitz-plus-Hankel matrix with
Toeplitz-plus-Hankel blocks and can be diagonalized by the
two-dimensional FCTs (which are tensor products of one-dimensional
FCTs) in $O(n^2\log n)$ operations provided that $h$ is quadrantally
symmetric, i.e., $h_{i,j}=h_{-i,j}=h_{i,-j}=h_{-i,-j}$ (refer to
\cite{NCT}).

For the anti-reflective BCs, we assume that the data
outside $u$ are an anti-reflection of the data inside $u$, i.e., a point outside the domain is anti-reflective to the closest boundary point first in one direction and then in the other direction. In particular,  we set
\[
\begin{array}{ccc}
u_{1-j,\phi}=2u_{1,\phi}-u_{j+1,\phi}, & u_{n+j,\phi}=2u_{n,\phi}-u_{n-j,\phi}, & \  {\rm for}\
1\le j\le m, 1\le\phi \le n,
\\
u_{\psi,1-j,}=2u_{\psi,1}-u_{\psi,j+1}, & u_{\psi, n+j}=2u_{\psi,n}-u_{\psi,n-j}, & \ {\rm for}\
1\le j\le m, 1\le \psi \le n.
\end{array}
\]
When both indices lie outside the range $\{1,\ldots, n\}$ (this
happens close to the $4$ corners of the given image), we set
\[
\begin{array}{cc}
u_{1-i,1-j} = 4 u_{1,1} - 2u_{1,j+1} -2u_{i+1,1} + u_{i+1,j+1}, \\
u_{1-i,n+j} = 4 u_{1,n} - 2u_{1,n-j} -2u_{i+1,n} + u_{i+1,n-j}, \\
u_{n+i,1-j} = 4 u_{n,1} - 2u_{n,j+1} -2u_{n-i,1} + u_{n-i,j+1}, \\
u_{n+i,n+j} = 4 u_{n,n} - 2u_{n,n-j} -2u_{n-i,n} + u_{n-i,n-j},
\end{array}
\]
for $1\leq i,j \leq m$. If  the blurring function (PSF) $h$ is
quadrantally symmetric, then the blurring matrix is a block
Toeplitz-plus-Hankel-plus-2-rank-correction matrix with
Toeplitz-plus-Hankel-plus-2-rank-correction blocks and can be
diagonalized by the two-dimensional anti-reflective transforms
(which are tensor products of one-dimensional anti-reflective
transforms $T_n$) in $O(n^2 \log n)$ real operations (see for
instance \cite{ADNS09}).
In the following we will assume a symmetric (quadrantally symmetric
in 2D) PSF since reflective and anti-reflective BCs can be
diagonalized by fast transforms only in such case. However, in the
nonsymmetric case, even if the blurring matrix can not be
diagonalized by fast transforms, the matrix-vector product can be
done again in $O(n^2\log n)$ by FFTs. Moreover, many practical blur
have the symmetry like the celebrated Gaussian blur widely used in
several contexts.

\section{Optimal Preconditioners with different Boundary Conditions}\label{prec-refl}
The optimal preconditioner for a matrix $A$ aims to find an approximation
which minimizes $\|B-A\|_F$ over all $B$ in a set of matrices for the matrix Frobenius
norm $\|\cdot\|_F$: the typical set of matrices is formed by considering an algebra of
matrices which are simultaneously diagonalized by a given unitary transform.
The main novelty in our context is represented by the fact that the anti-reflective matrices with
symmetric PSFs form a commutative algebra associated to a non-unitary transform. The latter
poses nontrivial difficulties that are treated in the sequel of the paper.
The optimal circulant preconditioner was originally given in \cite{TC88}.
The optimal sine transform preconditioner was presented in \cite{CNW96}.
The optimal cosine transform preconditioner was provided in \cite{CCW95}.
In this section, we construct the optimal reflective BCs preconditioner
and the optimal anti-reflective BCs preconditioner for (\ref{fp})  and the optimal reblurring preconditioner for the reblurring equation (\ref{fpp})  below instead of (\ref{fp}).  Some of these preconditioning techniques are inspired from the idea proposed in \cite{CCW95,CCW99} for zero-Dirichlet BCs.

\subsection{One-dimensional Problems}\label{sec:cos1d}
For the one-dimensional problems, we assume a symmetric and normalized PSF. Suppose that we impose the reflective BCs on $H$ and the zero Neumann BCs on $L(u^k)$. In this case, we propose the following reflective BCs preconditioners for (\ref{fp}).  Let $C_n$ be the $n$ dimensional discrete cosine transform  with entries
\[
[C_n]_{i,j} = \sqrt{\frac{2-\delta_{j1}}{n}} \cos\left(\frac{(2i-1)(j-1)\pi}{2n}\right),\quad i,j=1,\ldots,n,
\]
where $\delta_{ij}$ is the Kronecker delta. The matrix $C_n$ is
orthogonal, i.e., $C_nC_n^T = I$.  Moreover, for any $n$-vector $w$, the matrix-vector
product $C_nw$ can be computed within
$O(n\log n)$ real operations by the FCT.
Define
\begin{equation*}\label{set:cn}
{\cal C} = \{C_n^T\Lambda C_n: \, \Lambda \mbox{ is a real $n$-by-$n$ diagonal matrix}\}.
\end{equation*}
For an $n$-by-$n$ matrix $A$, the optimal cosine transform preconditioner is
\[c(A)= \argmin_{B \in \, {\cal C}}\|B-A\|_F.\]
The operator $c(\cdot)$ is linear,
preserves positive definiteness, and compresses every unitarily invariant norm. For the specific
cosine algebra and a general convergence theory
based on the Korovkin theorems, we refer to \cite{CCW99,koro1}.
As in \cite{CCW99} for Dirichlet BCs, the optimal cosine transform preconditioner (i.e., the optimal reflective BCs preconditioner) for (\ref{fp}) can be defined as
\begin{equation}\label{eq:R}
    R=H^*H+\alpha\, c(L(u^k)).
\end{equation}
We note that $R = c(A_{u^k})$ since $H^*=H \in {\cal C}$.
Spectral properties of the preconditioner will be discussed in Section \ref{sect:spectr}.
Here, we only note that $c(L(u^k))$ is not an optimal preconditioner for $L(u^k)$
if the coefficient $(|\nabla u|^2+\beta^2)^{-1/2}$ has large variation.
In such case a diagonal scaling is necessary to obtain an effective preconditioner
like $\diag(L(u^k))^{1/2}c(L(u^k))\diag(L(u^k))^{1/2}$, where
$\diag(L(u^k))$ is the diagonal matrix whose diagonal entries are the same as that of $L(u^k)$ \cite{glt-vs-fourier}.
We note that the coefficient matrix in (\ref{fp}) is the sum of two operators. To avoid the possibly large fluctuation in the coefficient of the operator in (\ref{fp}),  we  define a reflective BCs preconditioner for  (\ref{fp}) by
$
D_R = D^\frac{1}{2}R D^\frac{1}{2},
$
where $R$ is given in (\ref{eq:R}) and
\begin{equation}\label{eq:D}
D \equiv I + \alpha\, \diag(L(u^k)).
\end{equation}
A further possibility is to employ a diagonal scaling for \eqref{fp}. As in \cite{CCW99}, we concern the scaled equation
\begin{equation}\label{eq:scalesys}
        \tilde{A}_{u^k}\tilde{u}^{k+1} \equiv\left(\tilde{H}^*\tilde{H}  +
\alpha \, \tilde{L}(u^k)\right)\tilde{u}^{k+1} = \tilde{H}^*v,
\end{equation}
where $\tilde{H}=HD^{-1/2}$, $\tilde{L}(u^k)=
D^{-1/2}L(u^k)D^{-1/2}$,  and
$\tilde{u}^{k}=D^{1/2}{u}^{k}$. Then, we propose the following reflective BCs preconditioner for (\ref{eq:scalesys})
$
R_D = \hat{H}^*\hat{H} + \alpha \, c(\tilde{L}(u^k)),
$
where $\hat{H}= H c(D^{-1/2})$. If $\Lambda_H$, $\Lambda_D$
and  $\Lambda_{\tilde{L}}$ denote the eigenvalue matrices  of $H$,
$c(D^{-1/2})$ and $c(\tilde{L}(u^k))$, respectively, then $R_D$ can be written as
\[
R_D=C_n^T(\Lambda_H^*\Lambda_H\Lambda_D^*\Lambda_D + \alpha\, \Lambda_{\tilde{L}})C_n.
\]
Next, we construct the anti-reflective BCs preconditioners for (\ref{fp}) under the anti-reflective BCs for $H$ and the Neumann BCs or anti-reflective BCs for $L(u^k)$.
Let $S_n$ be the $n$ dimensional discrete sine transform of type I with entries as in (\ref{eq:lambda}).
Then, $S_n$ is orthogonal and symmetric, i.e., $S_n^T=S_n$
and $S_n^2 = I$. Moreover, for any $n$ dimensional vector $w$, the
matrix-vector product $S_nw$ can be computed in  $O(n\log n)$ real
operations by the FST. Define
$
\tau = \{S_n\Lambda S_n: \, \Lambda \mbox{ is a real diagonal matrix of order $n$}\}.
$
Let $\sigma(z):= (z_2,\ldots,z_n,0)^T$ with $z=(z_1,\ldots,z_n)^T$.
Let $\mathcal{T}(z)$  be the $n$-by-$n$ symmetric Toeplitz matrix whose first column is $z$
and $\mathcal{H}(z,Jz)$  be the $n$-by-$n$ Hankel matrix whose first and last column
are $z$ and $Jz$, respectively. It was shown that for any $B\in\tau$,
there exists $z=(z_1,\ldots,z_n)^T\in \mathbb{R}^n$ such that \cite{CNW96}
$
B=\mathcal{T}(z) - \mathcal{H}(\sigma^2(z),J\sigma^2(z)).
$
For an $n$-by-$n$ matrix $A$, the optimal sine preconditioner is
\begin{equation}\label{eq:s}
    s(A) = \argmin_{B\in\tau}\|B-A\|_F.
\end{equation}
The construction of $s(A)$ requires only $O(n^2)$ operation for a general matrix $A$ and $O(n)$ operation for a banded matrix $A$.
Furthermore, $s(\cdot)$ is linear,
preserves positive definiteness, and compresses any unitarily invariant norm (see \cite{CNW96,koro1}).

Now, we define an optimal sine transform based  preconditioner (i.e., the so-called anti-reflective BCs preconditioner) for (\ref{fp}) by
\begin{equation}\label{eq:M}
    M=\hat{s}(H)^*\hat{s}(H)+\alpha\, \hat{s}(L(u^k))
\end{equation}
in the sense that, for any $n$-by-$n$ matrix $A$, $\hat{s}(A)$ is given by
\begin{equation}\label{eq:hats}
    \hat{s}(A) = \argmin_{B\in\hat{\tau}}\|B-A\|_F,
\end{equation}
where
$
\hat{\tau} = \left\{\hat{S}_n\Lambda \hat{S}_n: \, \mbox{$\Lambda$ is a real diagonal matrix of order $n$ and $\hat{S}_n := \diag(1,S_{n-2},1)$}\right\}.
$
\begin{proposition}
Given an $n$-by-$n$ matrix $A$, we have
\[
\begin{array}{l}
\hat{s}(A) =\left[
                \begin{array}{ccc}
                  A(1,1) & 0 & 0 \\
                  0 & s(A(2:n-1,2:n-1)) & 0 \\
                  0 & 0 & A(n,n)
                \end{array}
              \right].
\end{array}
\]
where  $\hat{s}(\cdot)$ and $s(\cdot)$ are defined in \eqref{eq:hats} and \eqref{eq:s},
respectively, and $A(2:n-1,2:n-1)$ is the submatrix of $A$ corresponding to rows indexed from $2$ to $n-1$ and columns from $2$ to $n-1$.
\end{proposition}
\begin{proof}
By unitary invariance of the Frobenius norm ($\hat{S}_n$ is unitary) we find $\|A-\hat{S}_n\Lambda\hat{S}_n\|_F = \|\hat{S}_nA\hat{S}_n-\Lambda\|_F$,
where $\Lambda$ is a diagonal matrix.
To conclude the proof, it is enough to observe that
\[
\diag(\hat{S}_n A \hat{S}_n) =
\left[
                \begin{array}{ccc}
                  A(1,1) & 0 & 0 \\
                  0 & \diag(S_{n-2}A(2:n-1,2:n-1)S_{n-2}) & 0 \\
                  0 & 0 & A(n,n)
                \end{array}
              \right].
\]
and that $s(A)=S_n\diag(S_nAS_n)S_n$.
\end{proof}

To reduce the potential fluctuations in the coefficient of the elliptic operator in
(\ref{fp}), based on the diagonal scaling $D$ in \eqref{eq:D}, we define a scaled anti-reflective BCs preconditioner for (\ref{fp}) by
$
D_M = D^{\frac{1}{2}}MD^{\frac{1}{2}},
$
where $M$ is defined in \eqref{eq:M}. Similarly, for the scaled equation in the form of \eqref{eq:scalesys},
we give the anti-reflective BCs preconditioner with diagonal scaling as follows
$
M_D = \hat{H}^*\hat{H} + \alpha \, \hat{s}(\tilde{L}(u^k)),
$
where $\hat{H}= \hat{s}(H)\hat{s}(D^{-1/2})$. If $\Lambda_H$,
$\Lambda_D$ and  $\Lambda_{\tilde{L}}$ denote the eigenvalue
matrices  of $\hat{s}(H)$, $\hat{s}(D^{-1/2})$ and
$\hat{s}(\tilde{L}(u^k))$, respectively, then $M_D$ can be written as
\[
M_D=\hat{S}_n(\Lambda_H^*\Lambda_H\Lambda_D^*\Lambda_D
+ \alpha\,  \Lambda_{\tilde{L}})\hat{S}_n.
\]

Finally, we consider the reblurring method with some new reblurring preconditioners under the anti-reflective BCs for $H$ and the Neumann BCs for $L(u^k)$. In Section \ref{BCs}, we have observed that anti-reflective BCs matrices can be diagonalized by the anti-reflective transform $T_n$. Hence, it is possible to define the anti-reflective
algebra
\begin{equation}\label{eq:ARalg}
    {\cal AR} = \{T_n\Lambda T_n^{-1}: \, \Lambda \mbox{ is a real diagonal matrix of order $n$}\}.
\end{equation}
Unfortunately, $H\in {\cal AR}$ but $H^* \not\in {\cal AR}$. However, in \cite{DES06}, it
was proposed to use a reblurring approach, i.e., to replace $H^*$ with $H^\prime$,
where $H^\prime$ is the matrix obtained by imposing anti-reflective
BCs to the  PSF rotated by 180 degrees.
Since the PSF is assumed to be symmetric, $H^\prime=H$ \cite{DS05}.
Therefore, instead of (\ref{fp}), one may solve the following equation {\rm
\cite{DES06}}
\begin{equation} \label{fpp}
A^\prime_{u^k}u^{k+1} \equiv\left(H^\prime H  + \alpha\,
L(u^k)\right)u^{k+1} =  H^\prime v,\quad k =0,1,\ldots
\end{equation}
by the PBiCGstab method \cite{V92} since $A^\prime_{u^k}$ is not symmetric.
In this case, a reblurring preconditioner for (\ref{fpp}) is given by
\[
P=H^\prime H+\alpha\, AR(L(u^k)).
\]
Here, for any $n$-by-$n$ matrix $A$, $AR(A)$ is defined by
\[
AR(A) :=\left[
                \begin{array}{ccccc}
                  z_1 + 2\sum_{k=2}^{n-2}z_k & 0 & \cdots & 0 & 0 \\
                  z_2 + 2\sum_{k=3}^{n-2}z_k & & &  & 0 \\
                  \vdots & & & & z_{n-2} \\
                  z_{n-3} + 2z_{n-2} & s(A(2:n-1,2:n-1)) & & & z_{n-3} + 2z_{n-2} \\
                  z_{n-2} & & & & \vdots \\
                  0 & & & & z_2 + 2\sum_{k=3}^{n-2}z_k \\
                  0 & 0 & \cdots & 0 & z_1 + 2\sum_{k=2}^{n-2}z_k
                \end{array}
              \right],
\]
where $z =(z_1,z_2,\ldots,z_{n-2})^T$ is such that
$
s(A(2:n-1,2:n-1)) =  \mathcal{T}(z) -\mathcal{H}(\sigma^2(z),J\sigma^2(z)).
$
We  only need form $s(A(2:n-1,2:n-1))$ for computing $AR(A)$.

We note that $AR(A)$ belongs to the algebra $\mathcal{AR}$ defined in \eqref{eq:ARalg}, where
$\Lambda$ is defined as in \eqref{eq:lambda}.
Therefore, a linear system $A u=v$ can be solved within $O(n \log n)$
real operations by using three FSTs.
\begin{remark}
In general, $AR(A) \neq \argmin_{B\in{\cal AR}}\|B-A\|_F$.
Moreover, we can not construct $\argmin_{B\in{\cal AR}}\|B-A\|_F$ in only $O(n^2)$ operations
by using the similar technique for computing $s(A)$ for a general matrix $A$ in {\rm \cite{CCW95}}.
Notice that $T_n$ in \eqref{eq:Tn} is $T_n=\hat{S}_n(I+U)$ and $T_n^{-1}=(I-U)\hat{S}_n$, where
\[
U = \left(
      \begin{array}{ccc}
        0& 0 & 0 \\
        S_{n-2}p & 0 & S_{n-2}Jp  \\
        0 & 0 & 0
      \end{array}
    \right).
\]
As in {\rm \cite{H92}}, we can compute the eigenvalue of $\mathcal{AR}(A)$ by using the diagonal entries of
$
\Psi = \hat{S}_n A \hat{S}_n,
$
However, it requires $O(n^2 \log n)$ operations to calculate the diagonal entries of $\Psi$.
\end{remark}

To reduce the potential fluctuations in the coefficient of the elliptic operator in
(\ref{fpp}), we define a diagonally scaled reblurring preconditioner for (\ref{fpp}) as follows
$
D_P=D^{1/2}PD^{1/2},
$
where $D$ is defined as the same form in (\ref{eq:D}).
For the scaled system
\begin{equation}\label{eq:scalereb}
        \tilde{A}^\prime_{u^k}\tilde{u}^{k+1} \equiv\left(\tilde{H}^\prime\tilde{H}  +
\alpha \, \tilde{L}(u^k)\right)\tilde{u}^{k+1} = \tilde{H}^\prime v,
\end{equation}
the reblurring preconditioned is given by
\[
P_D = AR(D^{-1/2}) \, {H}^\prime H\, AR(D^{-1/2}) + \alpha \, AR(\tilde{L}(u^k)).
\]
If $\Lambda_H$, $\Lambda_D$, and
$\Lambda_{\tilde{L}}$ denote  the eigenvalue matrices  of $H$,
$AR(D^{-1/2})$, and $AR(\tilde{L}(u^k))$,
respectively, then the preconditioner $P_D$ can be written as
\[
P_D=T_n(\Lambda_{H}^*\Lambda_H\Lambda_D^*\Lambda_D
+ \alpha\, \Lambda_L)T_n^{-1}.
\]

A further possibility is the use of anti-reflective BCs for $L(u^k)$. This implies that the coefficient matrix in the linear equation (\ref{fpp}) is closer to the preconditioner. Consequently, a faster convergence
and a lower global cost have to be expected. The latter choice is in fact considered in the numerics.

We comment on the cost of constructing $X_D$, $X\in\{R,M,P\}$  and of
each PCG/ PBiCGstab iteration.  We note that $L(u^k)$ is a banded matrix. Therefore,
computing $c(L(u^k))$, $\hat{s}(L(u^k))$, and $AR(L(u^k))$ needs only $O(n)$ operations \cite{CCW95,CNW96}.  At each PCG/ PBiCGstab iteration, we need to calculate the matrix-vector
product $\tilde{A}_{u^k}w$ and $\tilde{A}'_{u^k}w$
and solve the system $X_D y =b$. The vector
multiplication $D^{-1/2}w$ can be computed in $O(n)$ operations
since $D^{-1/2}$ is a diagonal matrix. $L(u^k)w$ can be done in $O(n)$ operations. For $H\in {\cal C}$ or $H\in {\cal AR}$, $Hw$, $H^*Hw$, and $H'Hw$ can be calculated in $O(n\log n)$ operations by few FCTs or FSTs plus lower order of computations. The system $X_D y = b$ can also be solved in $O(n\log n)$ operations. Therefore, the total cost of each PCG/ PBiCGstab
iteration is bounded by $O(n\log n)$.
\subsection{Two-dimensional Problems}\label{sec:cos2d}
We can extend the results in Subsection \ref{sec:cos1d} to two-dimensional image deblurring problems with different BCs. In the two-dimensional case, we assume that the PSF is quadrantally symmetric and normalized. When one imposes the reflective BCs on $H$ and the zero Neumann BCs on $L(u^k)$, the blurring matrix $H$ is  a block Toeplitz-plus-Hankel matrix with Toeplitz-plus-Hankel blocks, which can be diagonalized by the
two-dimensional FCTs  in $O(n^2\log n)$ operations \cite{NCT}.
For an $n^2$-by-$n^2$ matrix $A$ in the form of
\begin{equation}\label{eq:blockA}
    A= \left(
     \begin{array}{cccc}
       A_{1,1} & A_{1,2} & \cdots & A_{1,n} \\
       A_{2,1} & A_{2,2} & \cdots & A_{2,n} \\
       \vdots & \ddots & \ddots & \vdots \\
       A_{n,1} & A_{n,2} & \cdots & A_{n,n}
     \end{array}
   \right),
\end{equation}
where $A_{i,j}$ are $n$-by-$n$ matrices, as defined in \cite{CCW99}, the Level-1 cosine transform preconditioner $c_1(A)$ is given by
\[
c_1(A)= \left(
     \begin{array}{cccc}
       c(A_{1,1}) & c(A_{1,2}) & \cdots & c(A_{1,n}) \\
       c(A_{2,1}) & c(A_{2,2}) & \cdots & c(A_{2,n}) \\
       \vdots & \ddots & \ddots & \vdots \\
       c(A_{n,1}) & c(A_{n,2}) & \cdots & c(A_{n,n})
     \end{array}
   \right),
\]
and then the Level-2 cosine transform preconditioner is $c_2(A)=Qc_1(Q^Tc_1(A)Q)Q^T$, where $Q$ be the permutation matrix which satisfies
$[Q^TAQ]_{i,j;k,l} = [A]_{k,l;i,j}$ for $1\le i,j \le n$ and $1\le k,l \le n$,
i.e., the $(i,j)$th entry of the $(k,l)$th block of $A$ is permuted to the $(k,l)$th  entry of the $(i,j)$th block.

For the two-dimensional linear equation (\ref{fp}), using $c_2(H) = H$, we define the optimal reflective BCs preconditioner for $A_{u^k}$ in (\ref{fp}) by
\begin{equation}\label{eq:R2}
    R=H^*H+\alpha\, c_2(L(u^k)).
\end{equation}

For eliminating the possibility of large variations in the coefficient of the elliptic operator
in (\ref{fp}),  we employ the same strategy as in Section \ref{sec:cos1d} by the diagonal scaling  in \eqref{eq:D}.
Therefore, the  scaled reflective BCs preconditioner is given by
\begin{equation}\label{eq:DR2}
D_R = D^\frac{1}{2}R D^\frac{1}{2}.
\end{equation}
where $R$ is defined in \eqref{eq:R2}.
Similarly, for the scaled system in  \eqref{eq:scalesys},
the reflective BCs preconditioner is given by
$
R_D = \hat{H}^*\hat{H} + \alpha \, c_2(\tilde{L}(u^k)),
$
where $\hat{H}= H c_2(D^{-1/2})$. Let $\Lambda_H$, $\Lambda_D$
and  $\Lambda_{\tilde{L}}$ denote the eigenvalue matrices  of $H$,
$c_2(D^{-1/2})$ and $c_2(\tilde{L}(u^k))$, respectively.
The preconditioner $R_{D}$ in (\ref{eq:DR2}) can be written as
\[
R_D=(C_n\otimes C_n)^T(\Lambda_H^*\Lambda_H\Lambda_D^*\Lambda_D + \alpha\, \Lambda_{\tilde{L}})(C_n\otimes C_n),
\]
and hence it is easily inverted by employing few FCTs in $O(n^2\log n)$ operations.

Next, we assume the anti-reflective BCs for $H$ and the Neumann BCs or anti-reflective BCs for $L(u^k)$. Then, we construct the anti-reflective BCs preconditioners for (\ref{fp}). For an $n^2$-by-$n^2$ matrix $A$ in \eqref{eq:blockA}, the Level-1 sine-based transform preconditioner $\hat{s}_1(A)$ is given by
\[
\hat{s}_1(A)= \left(
     \begin{array}{cccc}
       \hat{s}(A_{1,1}) & \hat{s}(A_{1,2}) & \cdots & \hat{s}(A_{1,n}) \\
       \hat{s}(A_{2,1}) & \hat{s}(A_{2,2}) & \cdots & \hat{s}(A_{2,n}) \\
       \vdots & \ddots & \ddots & \vdots \\
       \hat{s}(A_{n,1}) & \hat{s}(A_{n,2}) & \cdots & \hat{s}(A_{n,n})
     \end{array}
   \right).
\]
By using the same proof technique of Theorem 3.3 in \cite{NCT}, we can easily show that  the Level-2 sine-based transform preconditioner $\hat{s}_2(A)$ is given by $\hat{s}_2(A) =Q\hat{s}_1(Q^T\hat{s}_1(A)Q)Q^T$.
Notice that the matrix $H$ is the  anti-reflective BCs matrix. Then, we design the  sine-based transform preconditioner for (\ref{fp}) by
\[
M=\hat{s}_2(H)^*\hat{s}_2(H)+\alpha\, \hat{s}_2(L(u^k)).
\]
By employing the diagonal scaling in \eqref{eq:D}, we define the  scaled anti-reflective BCs preconditioner $D_M = D^{1/2}MD^{1/2}$ for  \eqref{fp}
and the anti-reflective BCs preconditioner
$
M_D = \hat{H}^*\hat{H} + \alpha \, c_2(\tilde{L}(u^k)),
$
where $\hat{H}= \hat{s}_2(H) \hat{s}_2(D^{-1/2})$,
for the two-dimensional system \eqref{eq:scalesys}.
Let $\Lambda_H$, $\Lambda_D$ and  $\Lambda_{\tilde{L}}$ denote the eigenvalue matrices
of $H$, $\hat{s}_2(D^{-1/2})$ and $\hat{s}_2(\tilde{L}(u^k))$, respectively.
Then, the  preconditioner $M_D$ takes the form
\[
M_D=(\hat{S}_n\otimes \hat{S}_n)(\Lambda_H^*\Lambda_H\Lambda_D^*\Lambda_D + \alpha\, \Lambda_{\tilde{L}})(\hat{S}_n\otimes \hat{S}_n),
\]
which is computationally attractive via FSTs since any matrix operation can be done within $O(n^2\log n)$ operations.

Finally, we assume the anti-reflective BCs for $H$ and the Neumann BCs for $L(u^k)$.
For an $n^2$-by-$n^2$ matrix $A$ defined in (\ref{eq:blockA}), the Level-1 reblurring preconditioner $AR_1(A)$ is given \nolinebreak by
\[
AR_1(A)= \left(
     \begin{array}{cccc}
       AR(A_{1,1}) & AR(A_{1,2}) & \cdots & AR(A_{1,n}) \\
       AR(A_{2,1}) & AR(A_{2,2}) & \cdots & AR(A_{2,n}) \\
       \vdots & \ddots & \ddots & \vdots \\
       AR(A_{n,1}) & AR(A_{n,2}) & \cdots & AR(A_{n,n})
     \end{array}
   \right).
\]
Using the same proof as in \cite[Theorem 3.3]{NCT},
we can easily show that the Level-2 reblurring preconditioner
$AR_2(A)$ is given by $AR_2(A) =QAR_1(Q^TAR_1(A)Q)Q^T$.
Now, we design the  reblurring preconditioner $AR_2(A'_{u^k})$ for the  linear equation (\ref{fpp}).
Since $H$ is the  anti-reflective BCs matrix, we define a  reblurring preconditioner for the linear equation (\ref{fpp}) as
$
P=H^\prime H+\alpha\, AR_2(L(u^k)).
$
Also, the  reblurring preconditioner with diagonal scaling is given by $D_P=D^{1/2}PD^{1/2}$
and the reblurring preconditioner for the two-dimensional scaled  linear system (\ref{eq:scalereb}) is
$
P_D = \hat{H}^\prime \hat{H} + \alpha \, AR_2(\tilde{L}(u^k)),
$
where $\hat{H}= H\cdot AR_2(D^{-1/2})$. Let $\Lambda_H$, $\Lambda_D$,
and  $\Lambda_{\tilde{L}}$ denote the eigenvalue matrices of $H$,
$AR_2(D^{-1/2})$, and $AR_2(\tilde{L}(u^k))$, respectively. Then, the
two-dimensional preconditioner $P_D$ can be written as
\[
P_D=(T_n\otimes T_n)(\Lambda_H^*\Lambda_H\Lambda_D^*\Lambda_D + \alpha\, \Lambda_{\tilde{L}})(T_n\otimes T_n)^{-1}.
\]
Again, these two-dimensional preconditioner shows interesting computational
features since the associated linear systems can be solved within $O(n^2\log n)$ operations.

\section{Asymptotic spectral analysis of the preconditioned sequences}\label{sect:spectr}

In order to study the effectiveness of the proposed preconditioners, we need the clustering analysis of the
spectrum. Also, localization of eigenvalues is
of interest when solving \eqref{eq:scalesys} via PCG or \eqref{eq:scalereb}  by PBiCGstab \cite{Axe}. Here is a useful definition \cite{glt-vs-fourier}
for sequences of matrices  $\{A_n\}$ where $A_n$ has size $d_n$, $n$ positive integer, and $d_k>d_q$ if $k>q$.

\begin{definition}\label{def-distribution}
A matrix sequence $\{A_n\}$ is said to be {\em distributed $($in the
sense of the eigenvalues$)$ as the pair $(\theta,G)$,} or {\em
have the distribution function $\theta$}, if,  for any $F\in \mathcal
C_0({\mathbb C})$, the following limit relation holds
\begin{equation}\label{distribution:sv-eig}
\lim_{n\rightarrow \infty}\frac{1}{n} \sum_{j=1}^{n}
F\left(\lambda_j(A_n)\right)=\frac1{\mu(G)}\,\int_G F(\theta(t))\,
dt,\qquad t=(t_{1},\ldots,t_{d}),
\end{equation}
where $\{\lambda_j(A_n)\}_{j=1}^n$ denote the eigenvalues of
$A_n$ and $\mu(\cdot)$ is the standard Lebesgue measure.
In that case we write $\{A_n\}\sim_{\lambda}(\theta,G)$.
\end{definition}

An interesting consequence of the equation \eqref{distribution:sv-eig} is that
$\{A_n\}\sim_{\lambda}(\theta,G)$ implies that most of the
eigenvalues are contained within any $\epsilon$-neighborhood of the
essential range of $\theta$. That is, the range of $\theta$ is
a cluster for the spectrum of the sequence $\{A_n\}$.

The main observation is that all the matrices considered so far
are low-rank perturbations of Toeplitz matrices or can be viewed as
extracted from Generalized Locally Toeplitz (GLT) sequences (see \cite{glt-vs-fourier}
and references therein and the seminal work
\cite{Tilli}). This observation is very important since every GLT
sequence has a symbol and this symbol is the spectral distribution
function in the sense of the latter definition. Furthermore, the
class of GLT sequences is an algebra of matrix-sequences. Hence, when making linear
combinations, products, or inverses (the latter operation only when the symbol does not
vanish on sets of positive measure), the result is a new GLT
sequence whose symbol can be obtained via the same operations on the
original symbols, as those performed on the matrices. Therefore, a particular case is that of
the preconditioned matrices can be seen again as extracted from a GLT
sequence whose symbol is the ratio of the symbols: here the numerator is the symbol of the original
matrix sequence and the denominator is the symbol of the preconditioning sequence.

In this section, according to Definition \ref{def-distribution} and since  we are interested in asymptotic estimates, we are forced to indicate explicitly the parameter $n$ which
uniquely defines the size of the associated matrix.
First, we discuss in detail the case of reflective BCs.
When considering $B_n=L(u^k)$, it is well-known \cite{glt-vs-fourier}
that
\[
\begin{array}{lcl}
\{B_n\} \sim_\lambda \left(a(x)w(t), G\right),& G = \Omega \times [0, 2\pi]^d,
\\[2mm]
a(x) = \frac{1}{\sqrt{|\nabla x|^2+\beta^2}}
& w(t) = \sum\limits_{i=1}^d(2-2\cos(t_i)).
\end{array}
\]
On the other hand, $c(B_n)\sim_\lambda \left( \overline a w(t), G\right)$,
where $\overline a$ is a constant and in fact it is the mean of the function $a(x)$:
$\overline a = \frac{1}{\mu(\Omega)}\int_{\Omega} a(x)\, dx$ .
The sequence $\{c(B_n)^{-1}B_n\}$ is clustered at one only if
the sequence $\{B_n-c(B_n)\}$ is clustered at zero.
Since $\{B_n-c(B_n)\}\sim_\lambda((a(x)-\overline a)w(t),G)$,
the optimal cosine preconditioner is effective only if the function
$a(x)$ has no large variation.
To obtain a clustering preconditioner, a diagonal scaling has
to be introduced. Indeed, the preconditioner $\diag(B_n)^{1/2}c(B_n)\diag(B_n)^{1/2}$ is such
that $\{\diag(B_n)^{1/2}c(B_n)\diag(B_n)^{1/2}\} \sim_\lambda (a(x)w(t),G)$ due to the
algebra stucture of GLT sequences,
and hence the preconditioned sequence is clustered at one.
In our case, the coefficient matrix
\[
A_n=H^*H+\alpha L(u^k)
\]
is the sum of an integral approximate operator and an approximate elliptic differential operator.
We note that $\{A_n\}\sim_\lambda (|\hat{h}(t)|^2+\alpha a(x)w(t),G)$,
where $\hat{h}$ is the symbol of the PSF defined in \eqref{eq:symb}
for the 1D case and similarly can be defined for $d>1$ (the entries of the
PSF are the Fourier coefficients of $\hat{h}$).
An effective preconditioner has to consider both terms which consitute the matrix $A_n$.
This is the aim of the preconditioner $R_n$ defined in \eqref{eq:R} and
\eqref{eq:R2} for the 1D and 2D case, respectively.
We have $\{R_n\}\sim_\lambda (|\hat{h}(t)|^2+\alpha \overline a  w(t),G)$
and so  $\{A_n-R_n\}\sim_\lambda ((a(x)-\overline a)\alpha w(t),G)$.
In this case, we can not apply a diagonal scaling to $c(L(u^k))$
because otherwise we loose the computational efficiency,
the matrix $H^*H+\alpha \diag(L(u^k))^{1/2}c(L(u^k))\diag(L(u^k))^{1/2}$
can not be diagonalized by discrete cosine transforms.
Therefore, we have to apply to $R_n$ a diagonal scaling
which should be act like $\diag(L(u^k))^{1/2}$ on $c(L(u^k))$,
while it should be no affect the term $H^*H$.
Unfortunately, since we have a diagonal scaling we can not apply
the diagonal scaling only to a term of the sum.
To balance the contribution of the two terms, the diagonal scaling
is defined by the matrix $D$ in \eqref{eq:D} which leads to the preconditioner $D_R$.
We have $\{(D_R)_n\}\sim_\lambda ((1+\alpha a(x))(|\hat{h}(t)|^2+\alpha w(t)),G)$
and hence the preconditioned sequence is not clustered at one,
even if for values of $\alpha$ used in the considered applications
it shows an optimal behaviour (see Figure \ref{fig:pcgn}).
We recall that the clustering is a useful property but it is not strictly necessary
for the optimality of the related preconditioned Krylov method: for instance in the
Hermitian positive definite case and when dealing with the PCG iterations, the spectral equivalence is sufficient.
Since $D_R^{-1}A_n$ is similar to $R^{-1}\tilde{A}_n$, with $\tilde{A_n}=D^{-1/2}A_nD^{-1/2}$,
the use of the preconditioner $D_R$ to the linear system \eqref{fp} is equivalent
to apply the preconditioner $R$ to the scaled linear system \eqref{eq:scalesys}.
However, the scaling of the linear suggest to use a cosine preconditioner
different from $R$ that is $R_D$ which is more effective for large values
of $\alpha$ (see numerical results in Section \ref{num2}).

\begin{remark}\label{rem:alpha}
For small values of $\alpha$, i.e., when few regularization is required,
the three preconditioners $R$, $D_R$ and $R_D$ have a similar behaviour.
Moreover, when $\alpha$ goes to zero the effectiveness of the proposed preconditioners
increases because the preconditioners and the original coefficient matrix  $A_n$ all tend to $H^*H$.
\end{remark}

For concluding this section, we note that in the case of anti-reflective BCs
similar considerations can be done. The main difference is when we consider
the reblurring strategy. However, using the results in \cite{GS07:jacobi-nonsymm},
the nonsymmetric case can be considered as well since the antisymmetric part
has trace norm (sum of all singular values) bounded by a pure constant independent of $n$.
Therefore the spectral distribution is governed by the symmetric part which is dominant
as discussed in Section 3.3 of \cite{ADS08}.

\section{Numerical Tests}\label{num2}

We will solve the problem (\ref{el}) by the fixed
point method \eqref{fp} with the operator $\ch$ approximated  by
using different BCs and with the matrix $\cl$ imposed by zero Neumann BCs or anti-reflective BCs. The algorithm was
implemented in {\tt MATLAB 7.10} and run on a PC Intel Pentium IV of
3.00 GHZ CPU. We shall show the effectiveness of the proposed
preconditioners for the signal/image deblurring and also give a
comparison of the quality of the restored signals/images with
different BCs.

In our test, for simplicity, we choose initial guess $u^0=v$
for the FP algorithm.  We shall solve  (\ref{fp})
by the PCG method when the Neumann BCs imposed on $L(u^k)$  and  solve  (\ref{fp})
by the PBiCGstab method when the anti-reflective BCs imposed on $L(u^k)$. Also, we  solve   (\ref{fpp}) by the PBiCGstab method. The initial guess for the PCG and PBiCGstab methods in $k$th FP
iteration is chosen to be the $(k-1)$th FP iterate.
The PCG and PBiCGstab iterations are
stopped  when the residual vector $r_k$ of the linear systems
(\ref{fp}) and (\ref{fpp}) at the $k$th iteration satisfies
$\|r_k\|_2/\|r_0\|_2< tol$, where $tol$ is set to $10^{-6}$
and $10^{-5}$ in the 1D and 2D case, respectively.

\subsection{1D case: Signal Deblurring}

\begin{figure}
\begin{center}
\begin{tabular}{cc}
  \epsfig{figure=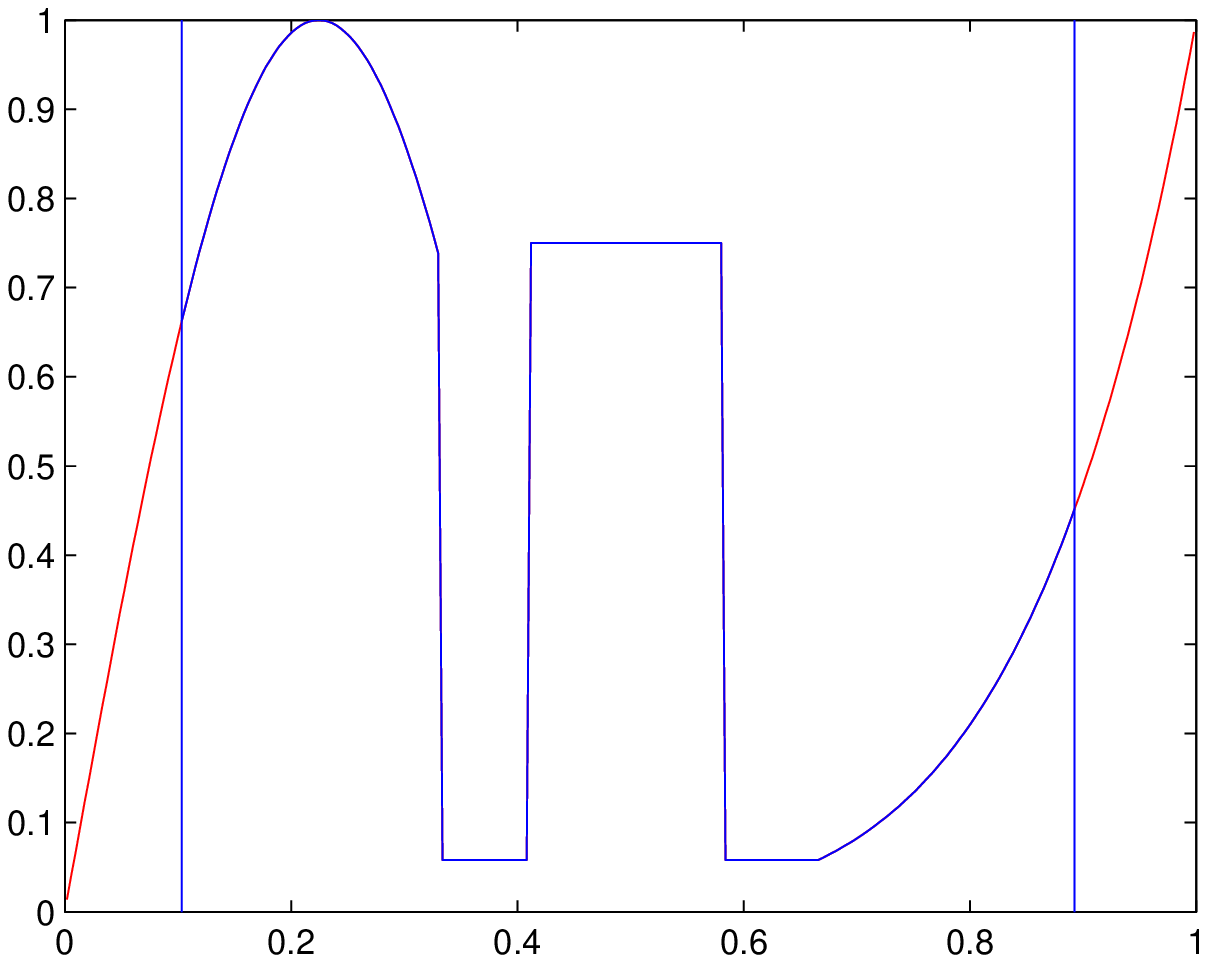,width=6cm} & \epsfig{figure=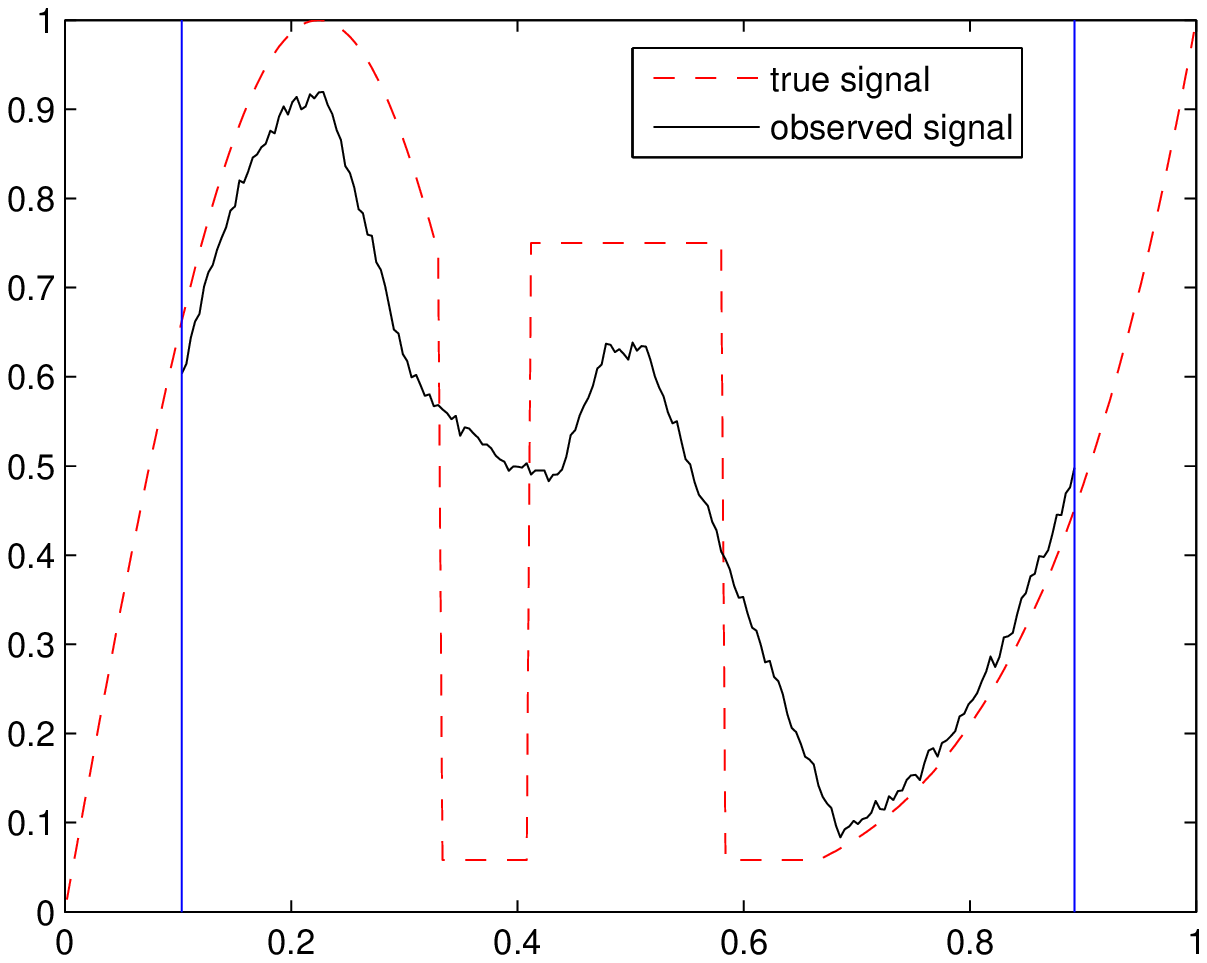,width=6cm} \\
  (a) True signal & (b) Observed signal\\
\end{tabular}
\end{center}
\caption{True and observed signals} \label{fig:s-of}
\end{figure}

In our experiments, we suppose the true signal $u$ is given as in
Figure  \ref{fig:s-of}(a). The two vertical lines shown in Figure
\ref{fig:s-of}(a) denote the field of view (i.e., $[0.1, 0.9]$) of
our signal and the signal outside the two vertical lines can be
approximated by different BCs. The true signal is
blurred by the symmetric out of focus  PSF:
\begin{equation} \label {of}
  h_i = \left\{ \begin{array}{ll}
          c & \mbox{if } |i| < m(n), \\
          0 & \mbox{otherewise},
        \end{array}\right.
  \end{equation}
where $c$ is the normalization constant such that $\sum_{i}h_i =1$
and $m(n)$ is the center of the  PSF which depends on $n$ so that
the restored signal lies in the interval $[0.1,0.9]$. A Gaussian noise $\eta$ with noise-to-signal ratio
$\|\eta\|_2/\|Hu\|_2$ is added to the blurred signal. We consider
the true signal is blurred by the out of focus PSF  and
then added the Gaussian noise with the noise levels $1\%$, i.e., $\|\eta\|_2/\|Hu\|_2=0.01$. Figure \ref{fig:s-of}(b) show the observed
signal.

We now show that the proposed preconditioners are effective for solving
(\ref{fp}) and (\ref{fpp}) with different BCs.
In our numerical experiment, the FP iteration is stopped when
$\|u^k-u^{k-1}\|_2/\|u^k\|_2<10^{-3}$.
We will concentrate on the performance of different
choices of preconditioners for various of parameters $\alpha$,
$\beta$, and $n$.

\begin{table}
 \begin{center}
   \begin{tabular}[c]{l|cccccc|cccccc}
     \hline
\textsf{PCG} & \multicolumn{6}{|c}{R} & \multicolumn{6}{|c}{AR+Sine+ZN $\cl$} \\ \hline
$\alpha$   & $N$  & $I$   & $D$   & $R$  & $D_R$& $R_D$ &  $N$ & $I$    & $D$    & $M$  & $D_M$& $M_D$\\   \hline
 $10^{-1}$ & $30$ & $269$ & $163$ & $73$ & $49$ & $45$  & $28$ & $221$  & $155$  & $60$ & $51$ & $36$ \\
 $10^{-2}$ & $37$ & $172$ & $107$ & $84$ & $37$ & $32$  & $24$ & $149$  & $ 94$  & $67$ & $33$ & $25$ \\
 $10^{-3}$ & $32$ & $99$  & $71$  & $63$ & $42$ & $36$  & $24$ & $ 80$  & $ 59$  & $56$ & $37$ & $31$ \\
 $10^{-4}$ & $19$ & $57$  & $56$  & $38$ & $36$ & $35$  & $19$ & $ 60$  & $ 58$  & $31$ & $30$ & $29$ \\
 $10^{-5}$ & $20$ & $74$  & $71$  & $23$ & $25$ & $24$  & $17$ & $ 47$  & $ 45$  & $20$ & $20$ & $19$ \\
 $10^{-6}$ & $11$ & $122$ & $122$ &  $7$ &  $7$ &  $8$  &  $8$ & $ 85$  & $ 86$  & $14$ & $14$ & $14$ \\
 \hline
\end{tabular}
\end{center}

\begin{center}
   \begin{tabular}[c]{l|cccccc|cccccc}
     \hline
\textsf{PBiCGstab} & \multicolumn{6}{|c|}{AR+Reblur+ZN $\cl$} & \multicolumn{6}{|c}{AR+Reblur+AR $\cl$}\\ \hline
$\alpha$   & $N$  & $I$   & $D$   & $P$  & $D_P$& $P_D$& $N$  & $I$   & $D$   & $P$  & $D_P$& $P_D$\\   \hline
 $10^{-1}$ & $27$ & $178$ & $105$ & $57$ & $53$ & $51$ & $24$ & $179$ & $109$ & $41$ & $46$ & $35$\\
 $10^{-2}$ & $24$ & $105$ & $59$  & $59$ & $25$ & $21$ & $23$ & $107$ & $60$  & $39$ & $23$ & $20$\\
 $10^{-3}$ & $23$ & $60$  & $44$  & $43$ & $27$ & $23$ & $22$ & $54$  & $37$  & $34$ & $22$ & $20$\\
 $10^{-4}$ & $20$ & $33$  & $31$  & $20$ & $18$ & $18$ & $20$ & $32$  & $30$  & $18$ & $16$ & $17$\\
 $10^{-5}$ & $25$ & $31$  & $33$  & $10$ & $ 9$ & $10$ & $20$ & $31$  & $32$  & $10$ & $10$ & $10$\\
 $10^{-6}$ &  $9$ & $78$  & $68$  & $4$  & $4$  & $ 4$ & $11$ & $56$  & $62$  & $4$  & $4$  & $4$\\
\hline
\end{tabular}
  \end{center}
  \caption{Average number of PCG/PBiCGstab iterations per FP step varying $\alpha$, with $n=203$ and $\beta=0.1$.}\label{table1-1}
\end{table}

\begin{table}
\begin{center}
   \begin{tabular}[c]{l|cccccc|cccccc}
     \hline
\textsf{PCG} & \multicolumn{6}{|c}{R} & \multicolumn{6}{|c}{AR+Sine+ZN $\cl$} \\ \hline
$\beta$    & $N$  & $I$   & $D$   & $R$   & $D_R$& $R_D$& $N$  & $I$    & $D$    & $M$   & $D_M$ & $M_D$  \\   \hline
 $10^{-3}$ & $31$ & $434$ & $245$ & $305$ & $298$& $125$& $24$ & $349$  & $200$  & $297$ & $226$ & $85$   \\
 $10^{-2}$ & $31$ & $218$ & $139$ & $149$ & $ 98$& $67$ & $24$ & $175$  & $112$  & $139$ & $ 72$ & $46$   \\
 $10^{-1}$ & $32$ & $99$  & $71$  & $63$  & $42$ & $36$ & $24$ & $ 80$  & $ 59$  & $ 56$ & $37$  & $31$   \\
 $10^{0}$  & $28$ & $39$  & $36$  & $21$  & $18$ & $20$ & $21$ & $ 35$  & $ 32$  & $ 19$ & $16$  & $15$   \\
\hline \end{tabular}
  \end{center}

 \begin{center}
   \begin{tabular}[c]{l|cccccc|cccccc}
     \hline
\textsf{PBiCGstab} & \multicolumn{6}{|c}{AR+Reblur+ZN $\cl$} & \multicolumn{6}{|c}{AR+Reblur+AR $\cl$}\\ \hline
$\beta$    & $N$  & $I$   & $D$  & $P$   & $D_P$& $P_D$& $N$  & $I$   & $D$   & $P$  & $D_P$ & $P_D$\\   \hline
 $10^{-3}$ & $23$ & $403$ & $254$& $361$ & $346$& $59$ & $22$ & $323$ & $207$ & $288$& $285$ & $57$\\
 $10^{-2}$ & $23$ & $157$ & $87$ & $130$ & $ 58$& $43$ & $22$ & $134$ & $77$  & $102$& $43$  & $28$\\
 $10^{-1}$ & $23$ & $60$  & $44$ & $43$  & $27$ & $23$ & $22$ & $54$  & $37$  & $34$ & $22$  & $20$\\
 $10^{0}$  & $20$ & $21$  & $20$ & $12$  & $11$ & $11$ & $20$ & $19$  & $18$  & $10$ & $ 9$  & $9$\\
\hline \end{tabular}
  \end{center}
  \caption{Average number of PCG/PBiCGstab iterations per FP step varying $\beta$, with $n=203$ and $\alpha=0.001$.}\label{table1-2}
\end{table}

In Tables \ref{table1-1} and \ref{table1-2}, we report the average number of iterations
per FP iteration, where $N$, $I$ and $D$ denote the number of FP steps,
no preconditioner and the diagonal scaling preconditioner, respectively.
According to Remark \ref{rem:alpha}, the effectiveness of the proposed preconditioners
increases when $\alpha$ decreases. Moreover, decreasing $\alpha$ all the proposed preconditioners
become equivalents, explicitly the PCG/PBiCGstat converges in about the same number of iterations.

We note that anti-reflective BCs usually require lesser steps and lesser PCG/ BiCGstab
iterations per FP step when compared with reflective BCs.
This shows that the improvement
in the model also leads to an improvement in the global computational complexity of the
numerical methods.
This is more evident for the optimal restoration since antireflective BCs require
a regularization parameter $\alpha$ smaller than the reflective BCs (Figures \ref{fig:tvrs}--\ref{fig:restored}).

Figure  \ref{fig:pcgn} describes the average PCG/PBiCGstab iterations per FP step
varying $n$. We note that the preconditioners with a diagonal scaling show an optimal behavior.

\begin{figure}
\begin{center}
    \begin{center}
    \begin{minipage}[c]{6cm}
        \centering
        \epsfig{figure=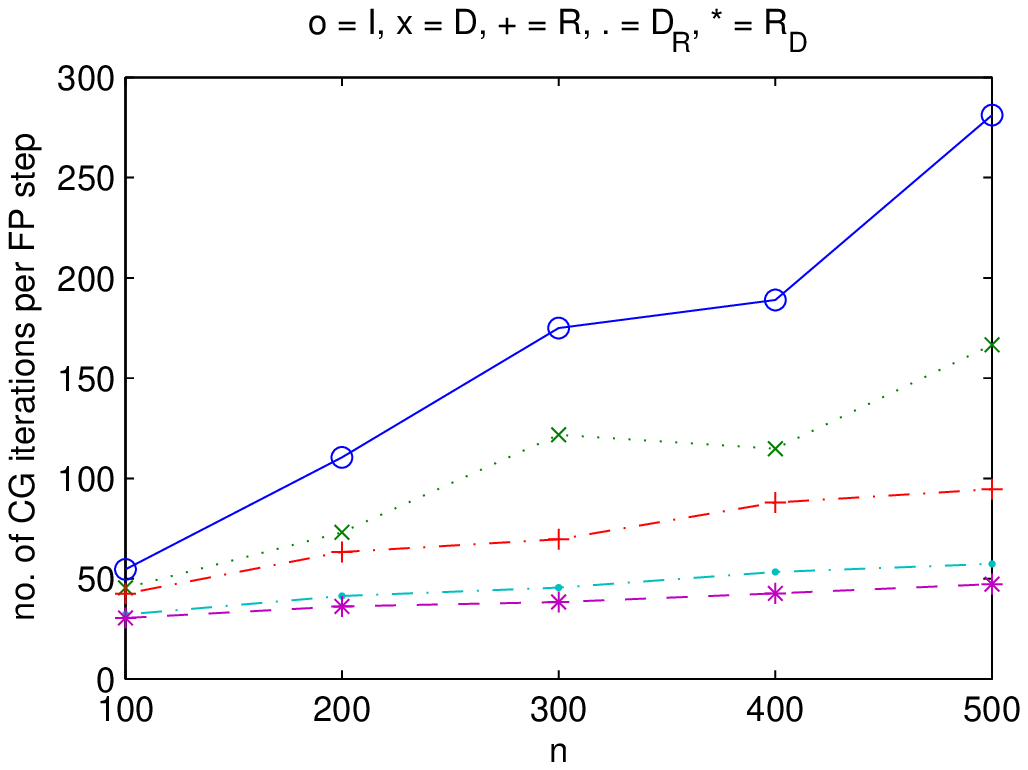,width=6cm}
        \small{(a)}
    \end{minipage}
    \begin{minipage}[c]{6cm}
        \centering
        \epsfig{figure=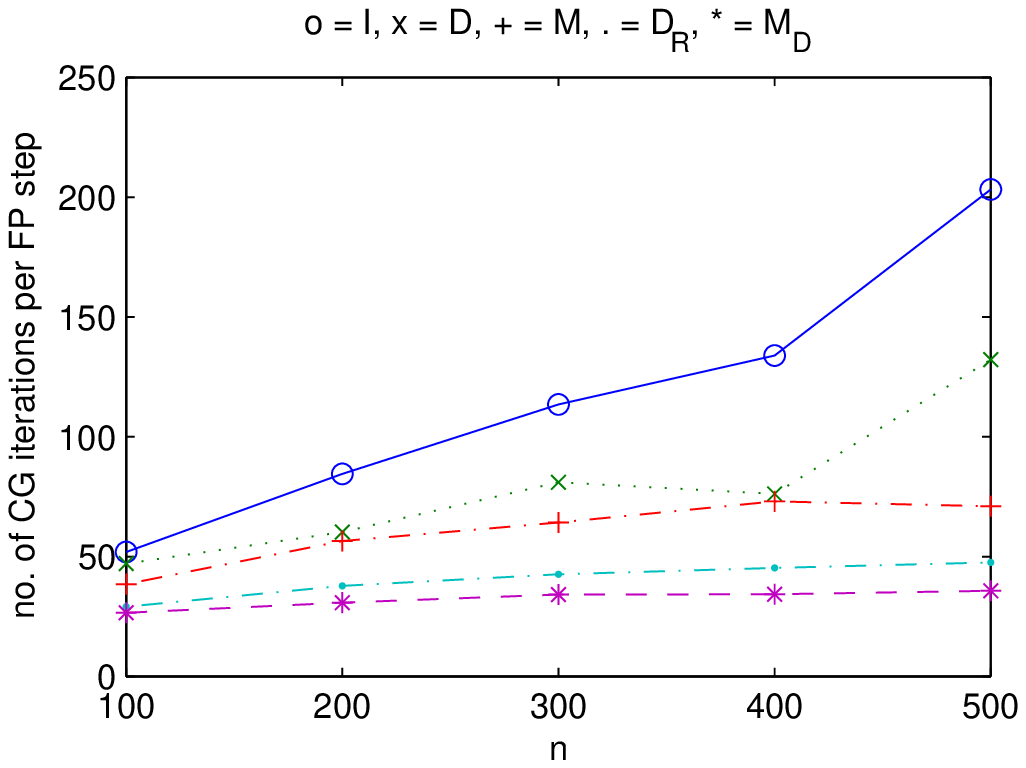,width=6cm}
        \small{(b)}
    \end{minipage} \\\vspace{0.5cm}
    \begin{minipage}[c]{6cm}
        \centering
        \epsfig{figure=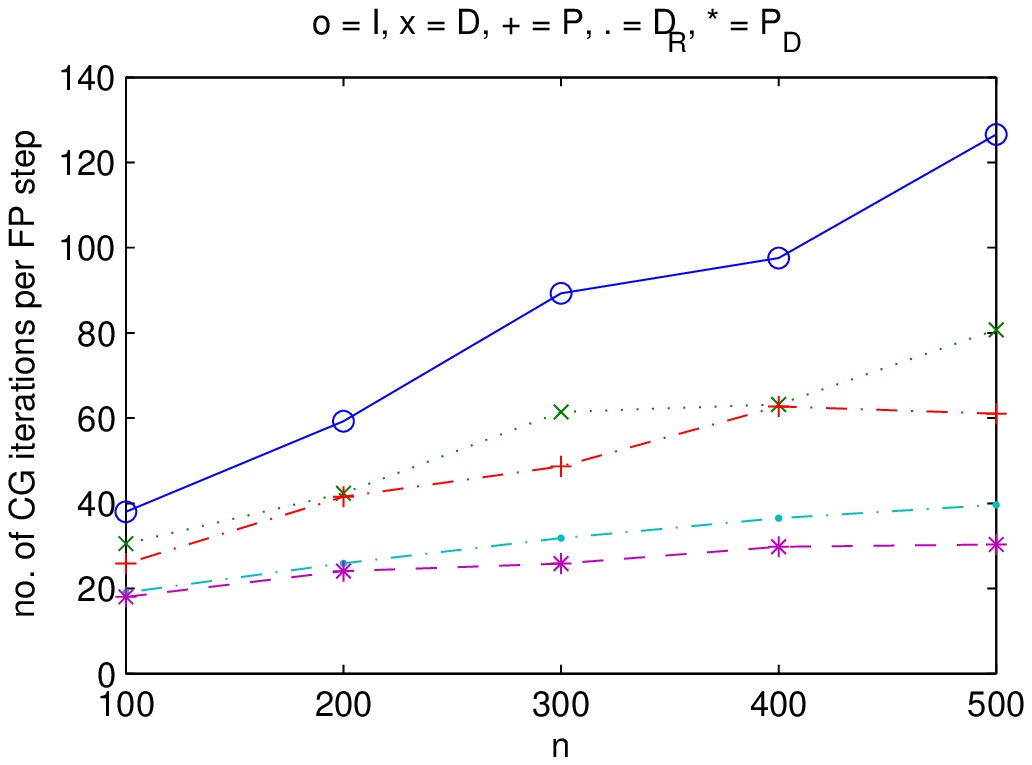,width=6cm}
        \small{(c)}
    \end{minipage}
    \begin{minipage}[c]{6cm}
        \centering
        \epsfig{figure=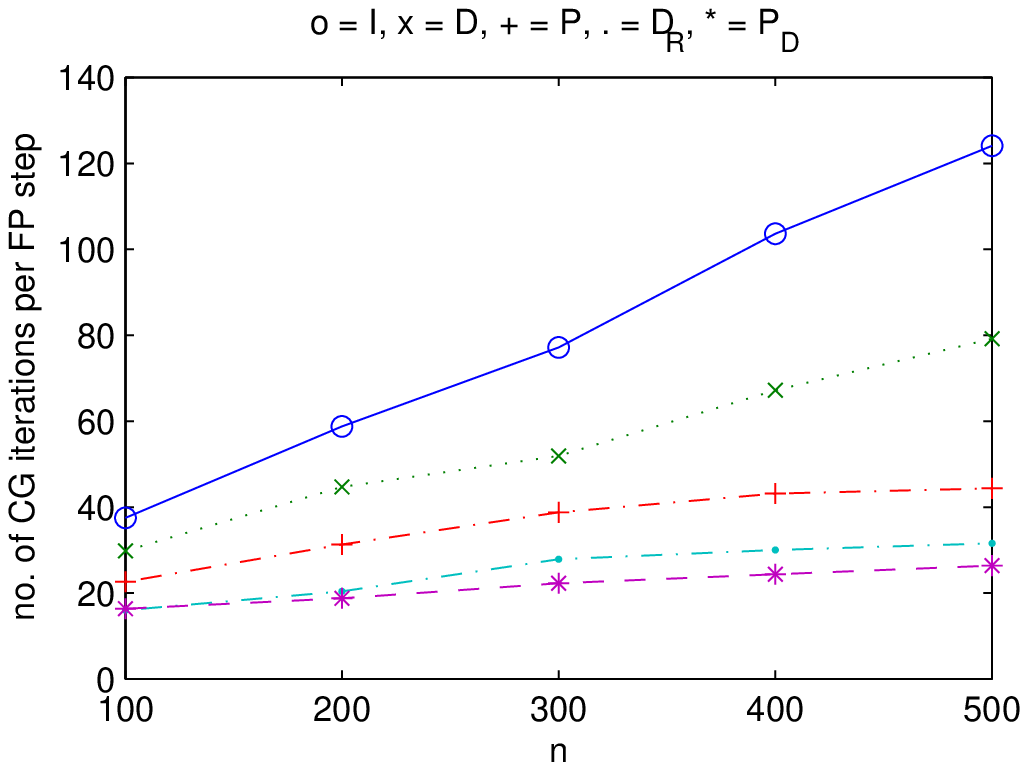,width=6cm}
        \small{(d)}
    \end{minipage}
    \end{center}
\end{center}
\caption{Average number of PCG/PBiCGstab iterations per FP step for
various $n$, with $\alpha =10^{-3}$ and $\beta= 0.1$.
(a) Reflective BCs.
(b) Anti-Reflective BCs with sine preconditioner.
(c) Anti-Reflective BCs with reblurring by imposing zero Neumann BCs on $\cl$.
(d) Anti-Reflective BCs with reblurring by imposing Anti-Reflective BCs on $\cl$.
} \label{fig:pcgn}
\end{figure}

In Figure \ref{fig:beta}, we present the restored signals for varying $\beta$,
e.g., by solving system (\ref{fp}) when anti-reflective BCs have been imposed.
As expected, the recovered signals become shaper when the value of $\beta$ is smaller,
the value $\beta=0.1$ gives a restoration sufficiently good anyway.

We can easily observe from Tables \ref{table1-1}-\ref{table1-2} and Figure
\ref{fig:pcgn} that the proposed preconditioners with a diagonal scaling
are the most effective preconditioner when varying parameters $\alpha$,
$\beta$, and $n$. Finally, we remark that
in all our tests, the proposed algorithm needs the same number of FP steps for the
no-preconditioner and  preconditioned cases and  $\|g(u^k)\|_2$ tends to
$O(10^{-5})$ or $O(10^{-6})$ at the final FP iterate.

\begin{figure}
\begin{center}
    \begin{minipage}[c]{6cm}
        \centering
        \epsfig{figure=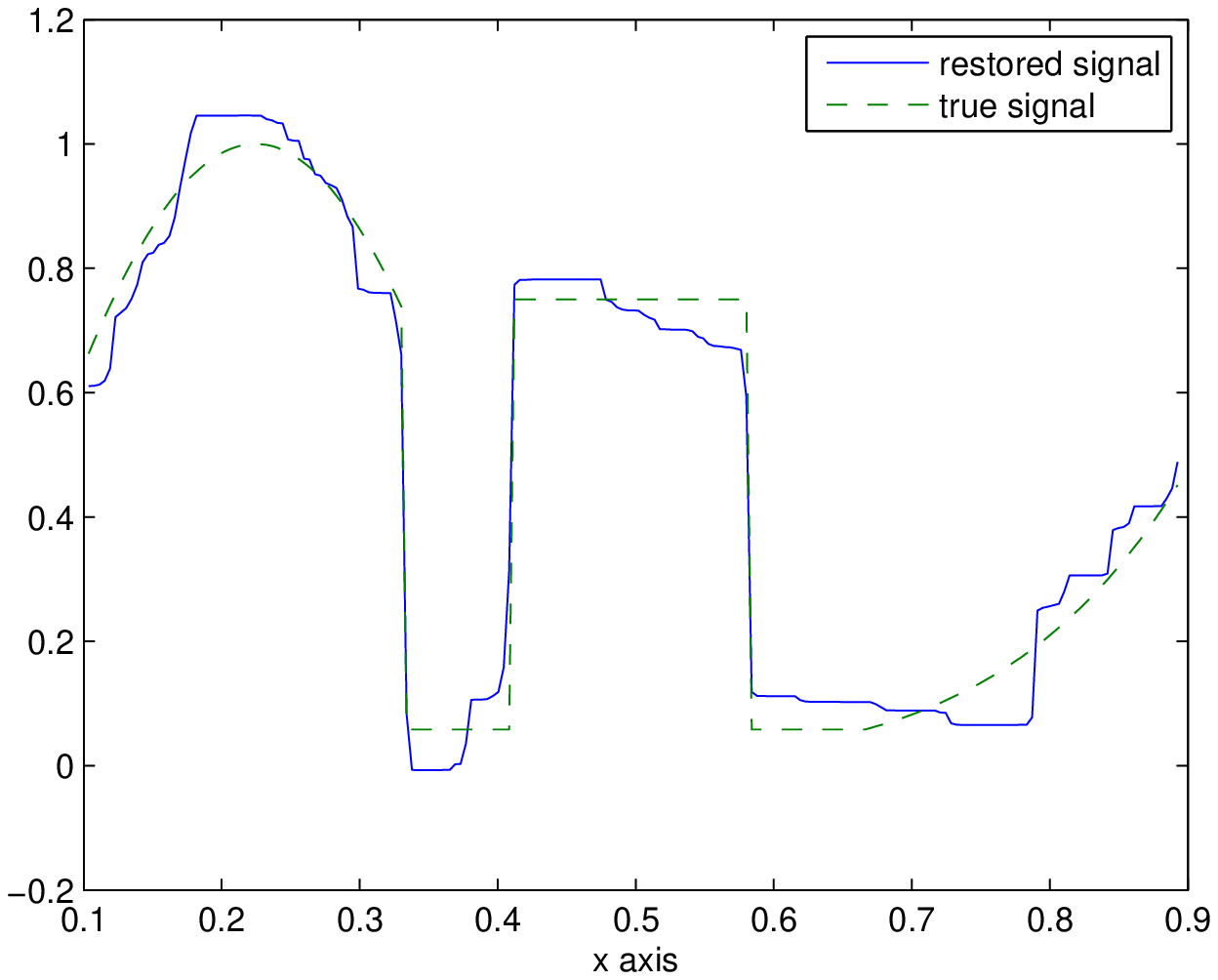,width=6cm}
        \small{$\beta=0.01$}
    \end{minipage}
    \begin{minipage}[c]{6cm}
        \centering
        \epsfig{figure=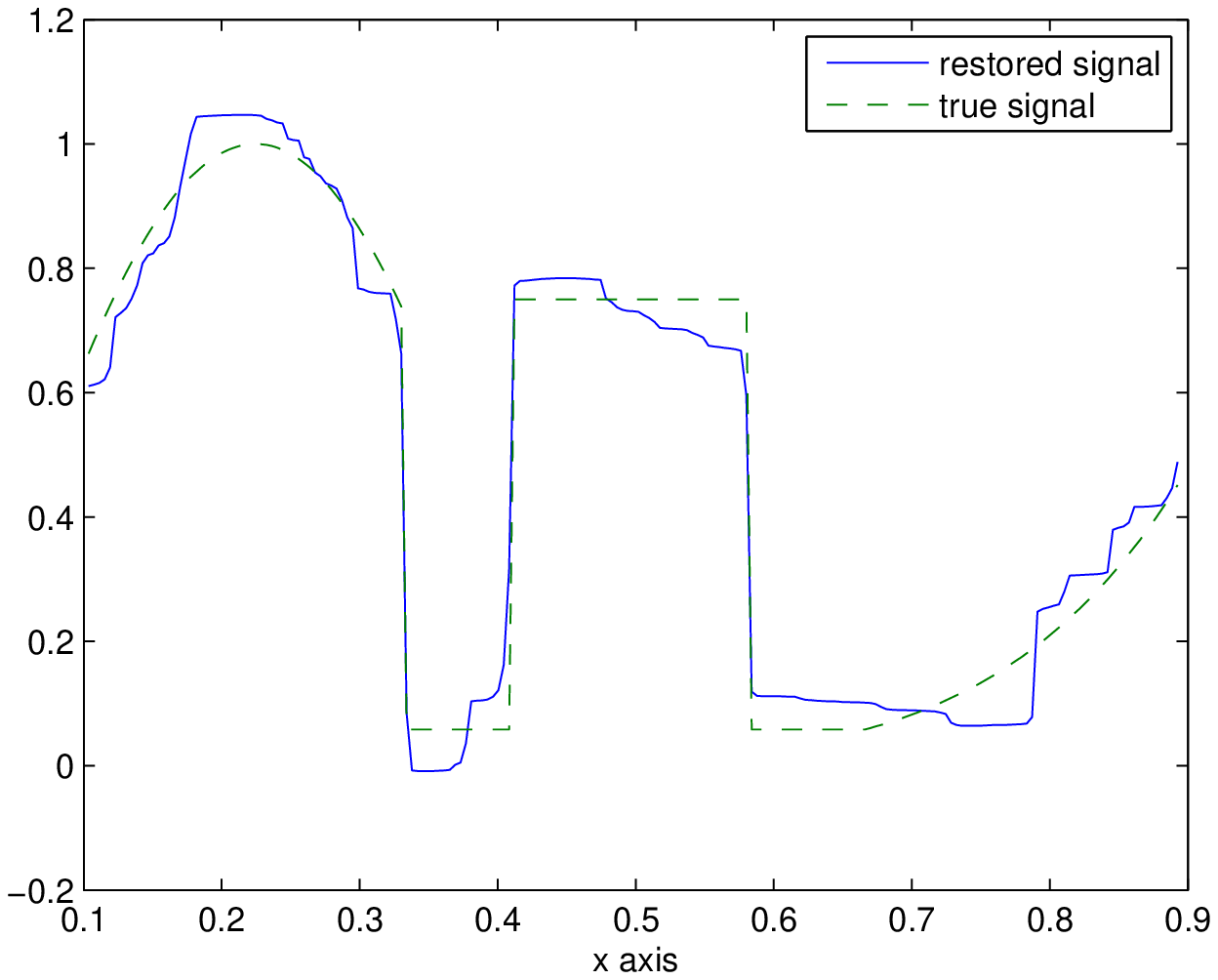,width=6cm}
        \small{$\beta=0.1$}
    \end{minipage}\\\vspace{0.5cm}
    \begin{minipage}[c]{6cm}
        \centering
        \epsfig{figure=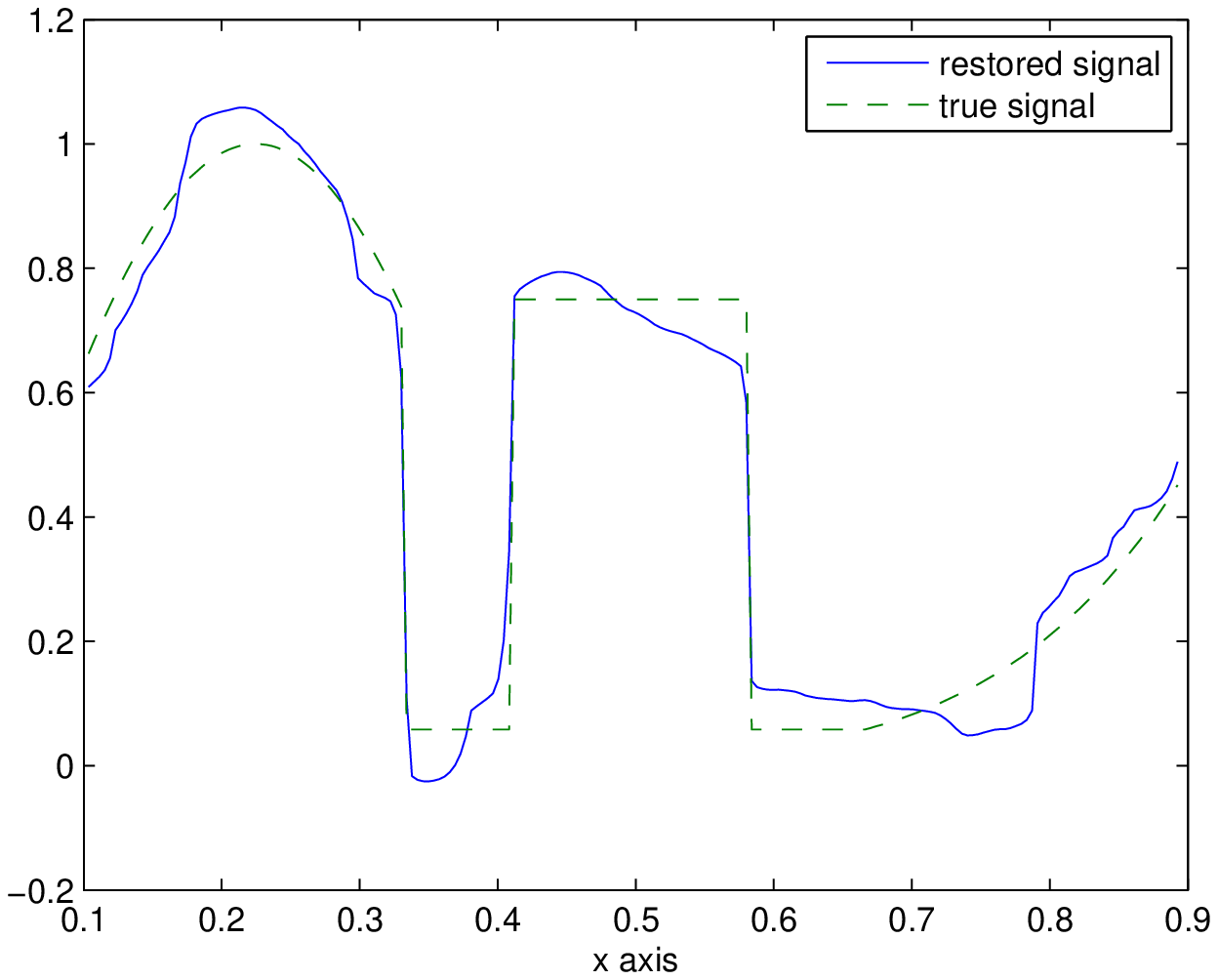,width=6cm}
        \small{$\beta=1$}
    \end{minipage}
    \begin{minipage}[c]{6cm}
        \centering
        \epsfig{figure=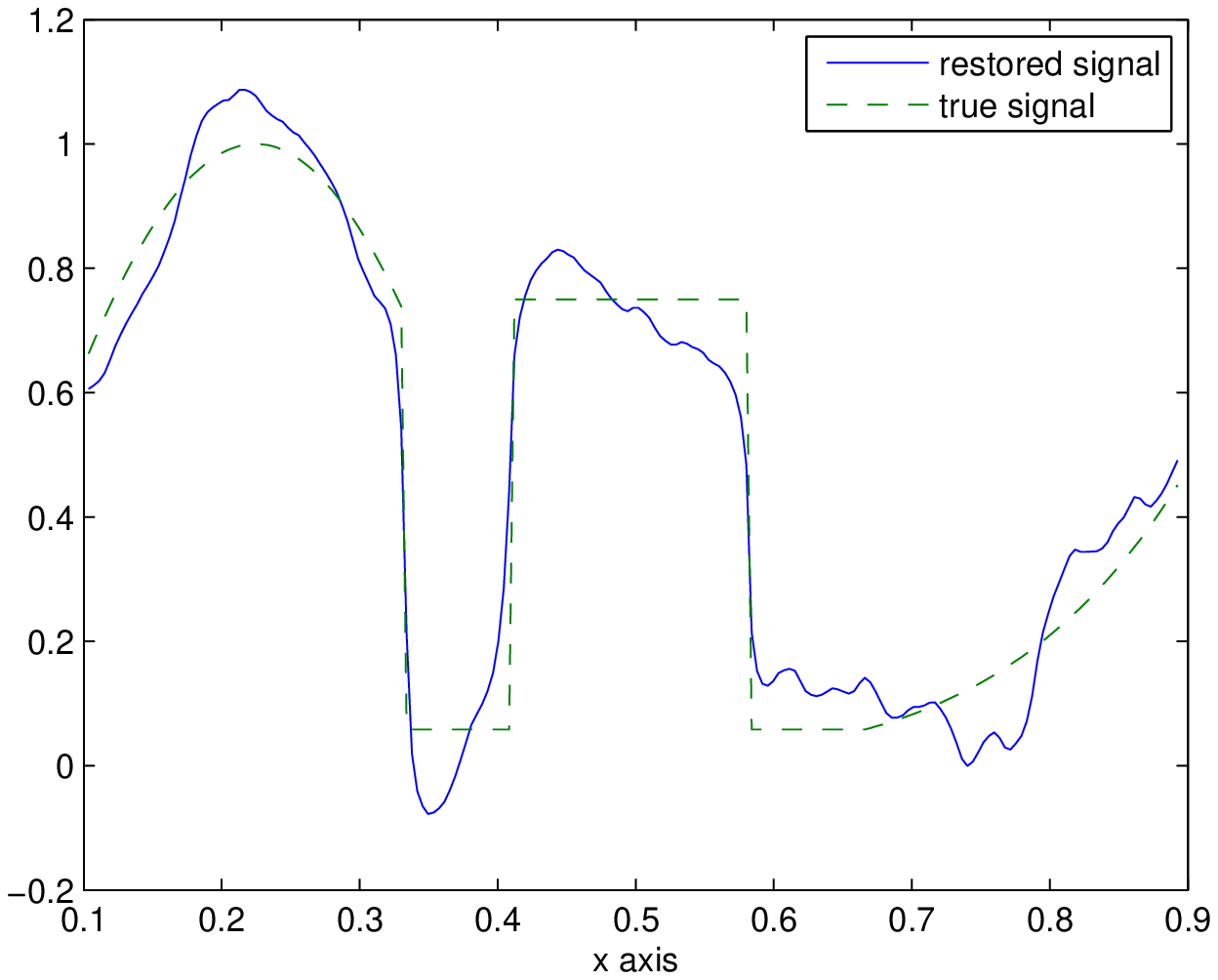,width=6cm}
        \small{$\beta=10$}
    \end{minipage}
\end{center}
\caption{Restorations for Anti-Reflective BCs based (\ref{fp}) with
$n=203$, $\alpha =10^{-3}$ varying $\beta$.}
\label{fig:beta}
\end{figure}

\begin{figure}
\begin{center}
\subfigure{\epsfig{figure=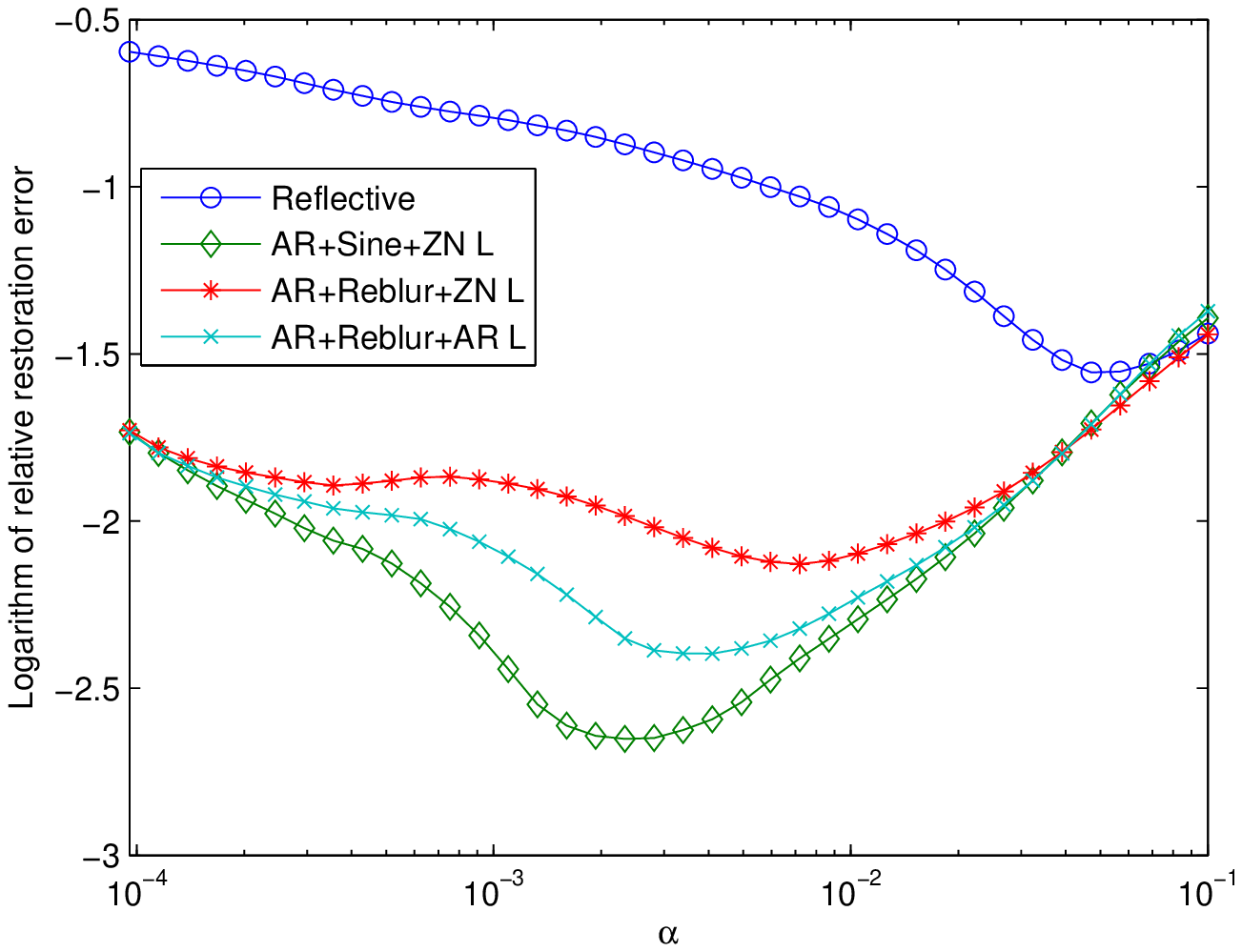,width=9cm}}
\end{center}
\caption{The RRE versus the regularization parameter $\alpha$ for different BCs.} \label{fig:tvrs}
\end{figure}

To check the quality of restored signals by using different BCs,
in Figure \ref{fig:tvrs} we show
the relative restoration error (RRE), $\|u_\alpha-u\|_2/\|u\|_2$, where $u_\alpha$ is
the total variation regularized solution of the true signal $u$,  versus the regularization
parameter $\alpha$. Figure \ref{fig:restored} gives the restored signals with optimal
value of the parameter $\alpha$, where  ${\tt \alpha_{opt}}$, {\tt Re.}, {\tt Fp.},  and ${\tt It.}$
denote the optimal value of the parameter $\alpha$, the minimal RRE, the number of FP steps, and the average number
of PCG/BiCGstab iterations per FP step, respectively.

From Figure \ref{fig:restored} we argue that anti-reflective BCs lead to the most
accurate restored signals with less significant ringing effects at the edges
and less PCG iterations per FP step, when compared with reflective BCs.
Moreover, thanks to the improvement in the model of the problem,
antireflective BCs require lesser regularization than reflective BCs.
This implies a smaller ${\tt \alpha_{opt}}$ and hence a small number of
PCG/BiCGstab iterations per FP step, while the number of FP iterations remains about the same.

\begin{figure}
\begin{center}
    \begin{minipage}[c]{6cm}
        \centering
        \epsfig{figure=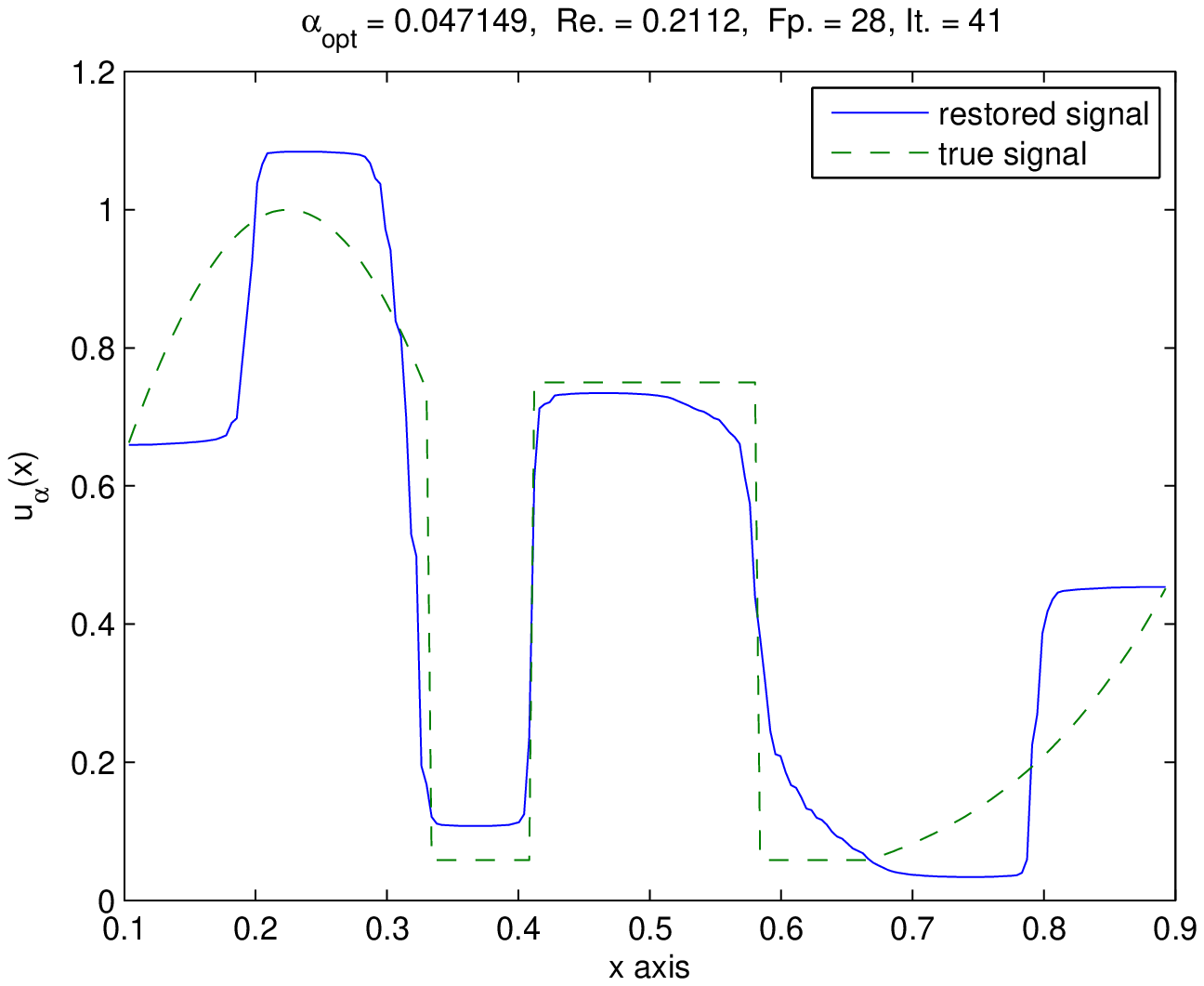,width=6cm}
        \small{Reflective}
    \end{minipage}
    \begin{minipage}[c]{6cm}
        \centering
        \epsfig{figure=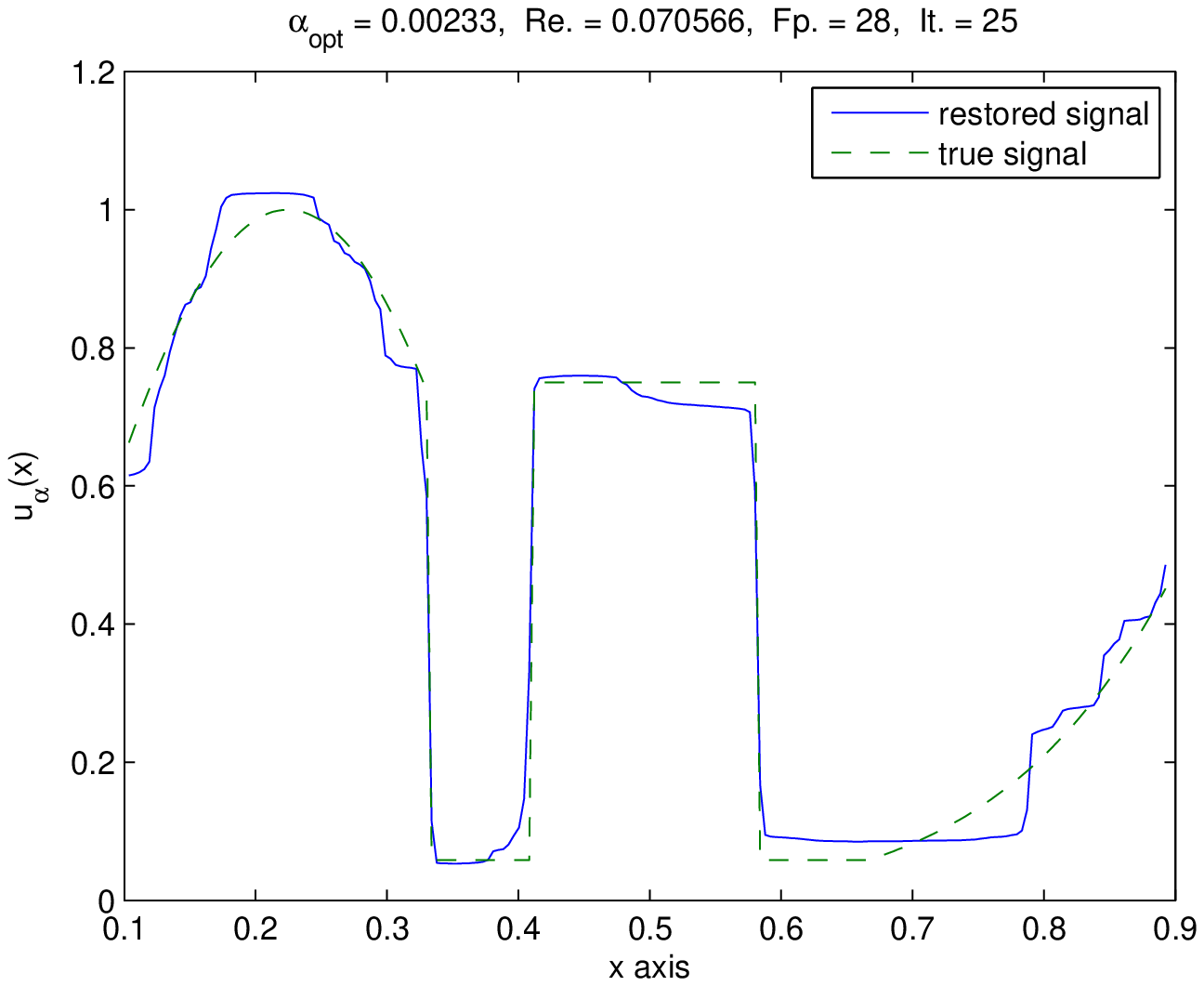,width=6cm}
        \small{AR+Sine+ZN $\cl$}
    \end{minipage}\\\vspace{0.5cm}
    \begin{minipage}[c]{6cm}
        \centering
        \epsfig{figure=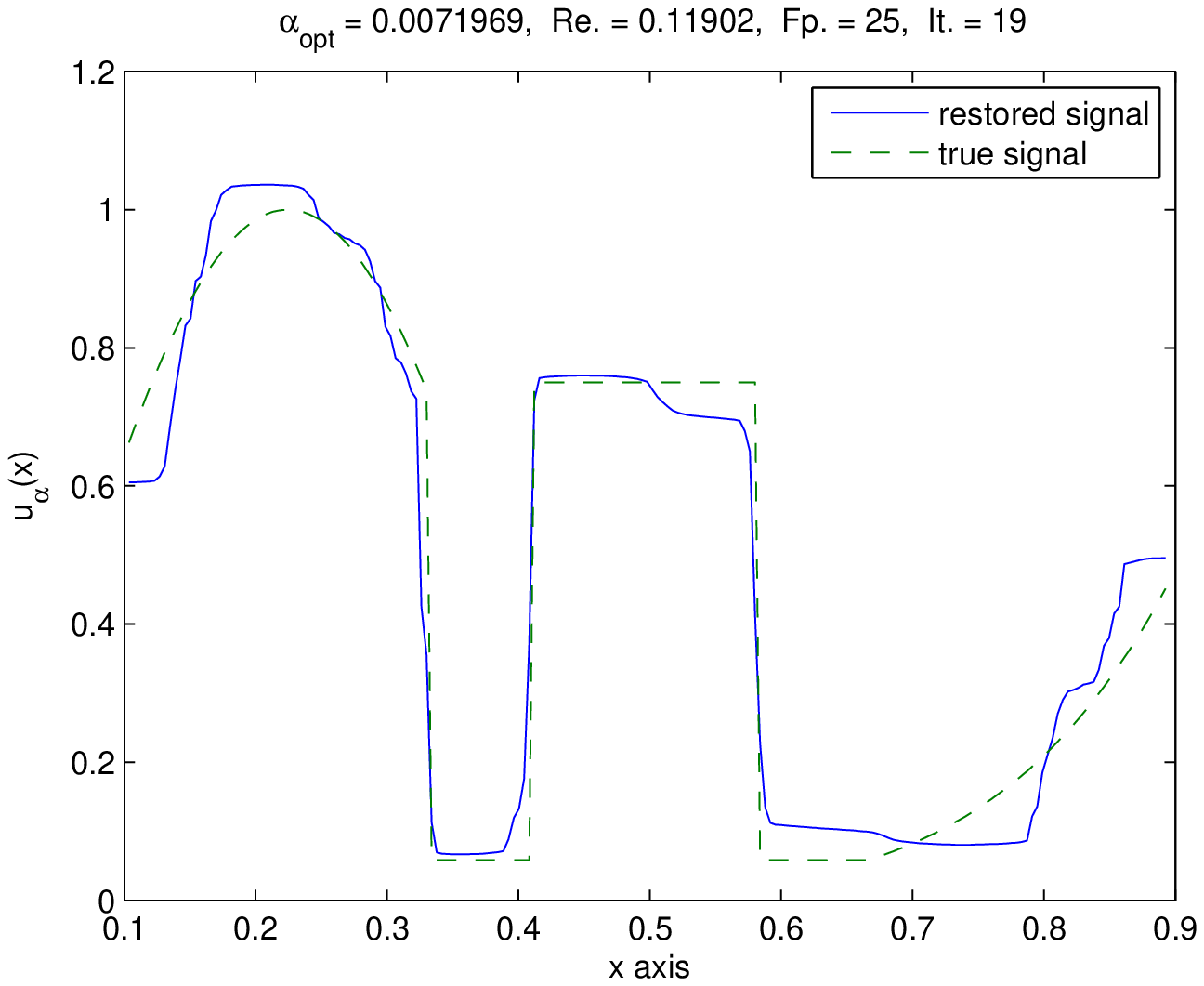,width=6cm}
        \small{AR+Reblur+ZN $\cl$}
    \end{minipage}
    \begin{minipage}[c]{6cm}
        \centering
        \epsfig{figure=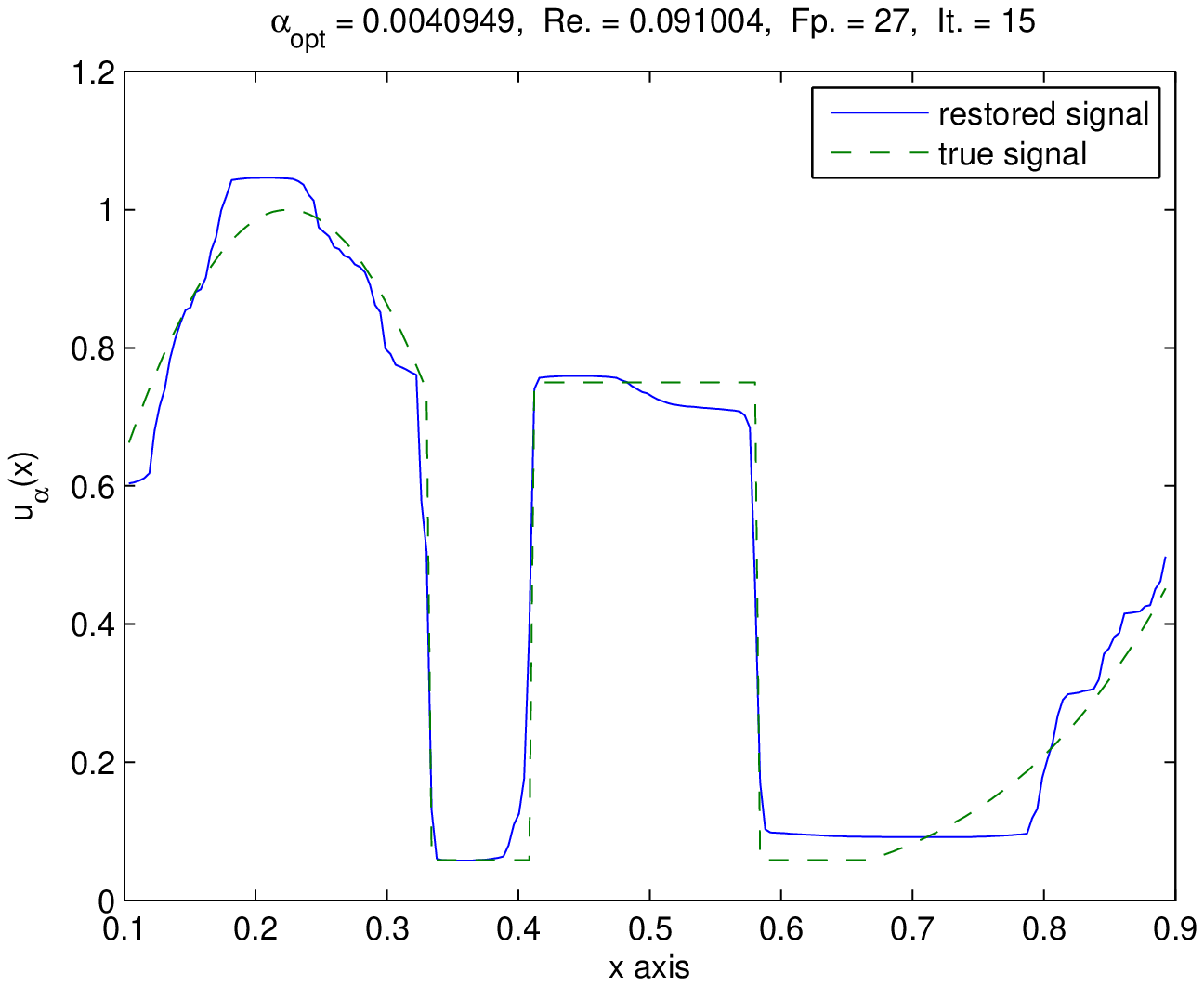,width=6cm}
        \small{AR+Reblur+AR $\cl$}
    \end{minipage}
\end{center}
\caption{Restored signals with different BCs. Here $n=203$ and $\beta=0.1$ } \label{fig:restored}
\end{figure}

\subsection{2D case: Image Deblurring}
In this section, we apply the proposed preconditioners to image restoration with different BCs.
Suppose the true images are blurred by the Gaussian blur and then suppose that a white Gaussian noise $\eta$
with the noise levels $0.1\%$ is added. Figure \ref{fig:im-g} shows the true and observed images.

\begin{figure}
\begin{center}
    \begin{minipage}[c]{6cm}
        \centering
        \epsfig{figure=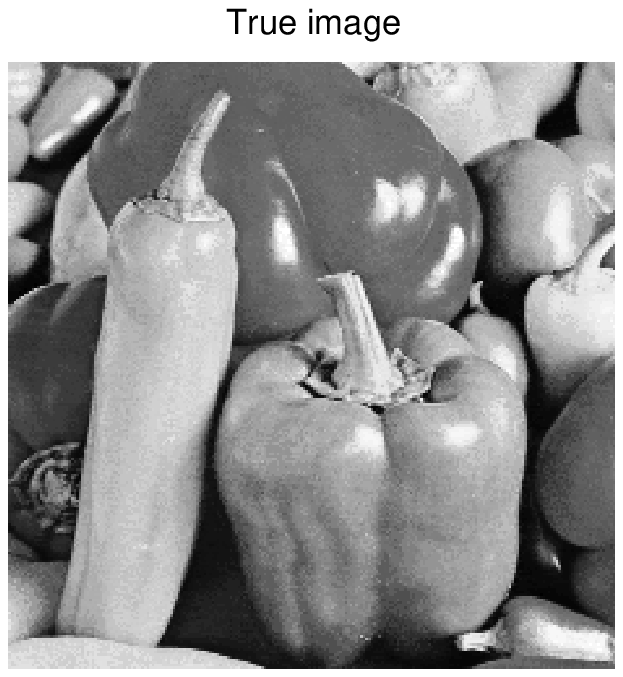,width=6cm}
    \end{minipage}\hspace{-1.2cm}
    \begin{minipage}[c]{6cm}
        \centering
        \epsfig{figure=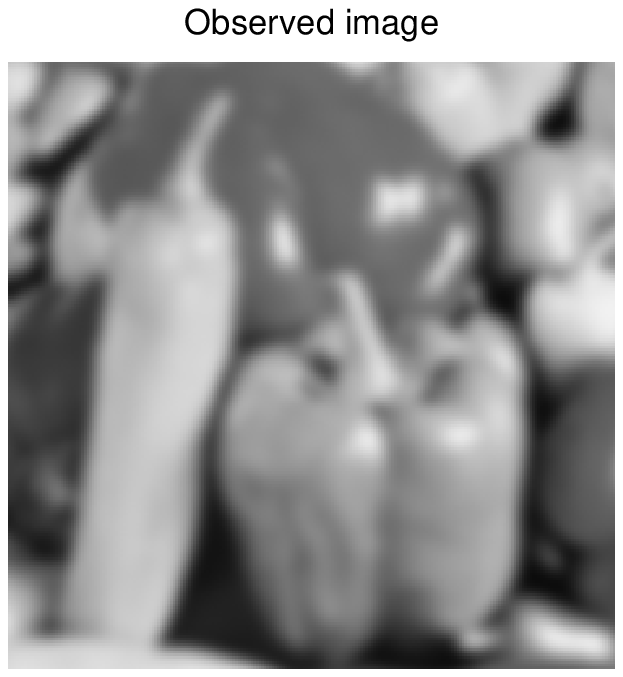,width=6cm}
    \end{minipage}
\end{center}
\caption{True and observed images} \label{fig:im-g}
\end{figure}

In our numerical tests, the FP iteration is stopped when
$\|u^k-u^{k-1}\|_2/\|u^k\|_2$ $ <10^{-4}$ and the maximal number
of FP steps is set to be $100$. We fix $\beta=0.01$ and only focus on the
performance of different choices of preconditioners for varying $\alpha$.

\begin{table}
  \begin{center}
   \begin{tabular}[c]{l|cccccc|cccccc}
     \hline
\textsf{PCG} & \multicolumn{6}{|c}{R} & \multicolumn{6}{|c}{AR+Sine+ZN $\cl$}  \\ \hline
$\alpha$  & $N$   & $I$    & $N$   & $D_R$& $N$   & $R_D$ & $N$  & $I$    & $N$   & $D_M$& $N$   & $M_D$ \\   \hline
$10^{0}$  & $60$  & $244$  & $60$  & $48$ & $60$  & $23$  & $60$ & $229$  & $60$  & $210$ & $61$  & $23$   \\
$10^{-1}$ & $41$  & $129$  & $46$  & $29$ & $33$  & $22$  & $40$ & $114$  & $40$  & $117$ & $28$  & $38$   \\
$10^{-2}$ & $13$  & $70$   & $13$  & $44$ & $11$  & $39$  & $12$ & $ 58$  & $9$  & $49$ & $ 8$  & $41$   \\
$10^{-3}$ & $ 3$  & $124$  & $ 2$  & $42$ & $ 2$  & $42$  & $ 2$ & $147$  & $2$  & $36$ & $ 2$  & $37$   \\
$10^{-4}$ & $ 2$  & $164$  & $ 1$  & $23$ & $ 1$  & $23$  & $ 1$ & $354$  & $1$  & $23$ & $ 1$  & $23$   \\
\hline
\end{tabular}
\end{center}

\begin{center}
\begin{tabular}[c]{l|cccccc|cccccc}  \hline
\textsf{PBiCGstab} &  \multicolumn{6}{|c|}{AR+Reblur+ZN $\cl$} & \multicolumn{6}{|c}{AR+Reblur+AR $\cl$}\\ \hline
$\alpha$  & $N$   & $I$  & $N$   & $D_P$& $N$   & $P_D$& $N$   & $I$   & $N$   & $D_P$& $N$  & $P_D$ \\ \hline
$10^{0}$  & $58$  & $52$ & $*$  & $*$ & $59$  & $13$ & $60$  & $52$  & $59$  & $50$ & $60$ & $7$   \\
$10^{-1}$ & $27$  & $27$ & $33$  & $16$ & $31$  & $7 $ & $26$  & $28$  & $31$  & $9$ & $30$ & $5$   \\
$10^{-2}$ & $10$  & $ 9$ & $10$  & $7$ & $10$  & $ 6$ & $10$  & $10$  & $12$  & $5$ & $10$ & $5$   \\
$10^{-3}$ & $ 2$  & $16$ & $2$  & $5$ & $ 2$  & $ 5$ & $ 3$  & $11$  & $2$  & $4$ & $ 2$ & $5$   \\
$10^{-4}$ & $ 1$  & $33$ & $1$  & $3$ & $ 2$  & $ 1$ & $ 1$  & $34$  & $1$  & $3$ & $ 2$ & $1$   \\
\hline
\end{tabular}
  \end{center}
  \caption{Average number of CG/BiCGstab iterations per FP step for varying $\alpha$. Here, $\beta=0.01$
  ($*$ means that the method does not converge).}\label{table2-1}
\end{table}

In Table \ref{table2-1} the number of iterations is displayed for solving (\ref{fp}) and
(\ref{fpp}) with different BCs and various values of $\alpha$, where $N$ and $I$ mean the number of FP iterations and
no preconditioner, respectively. Here, we only give the average number of CG/BiCGstab iterations per FP step.
Table \ref{table2-1} suggests that the preconditioners $X_D$, with $X \in \{R, M, P\}$, are very effective
matrix approximations for all values of  $\alpha$, while the preconditioners $D_X$ do not work well for large
values of $\alpha$ especially for antireflective BCs. However, for antireflective BCs a good choice for $\alpha$
is in the interval $[10^{-3}, 10^{-2}]$ and in this case both choices $X_D$ and $D_X$ have a similar behaviour.
We note that the number of FP iterations decrease with $\alpha$, so if we have a good model that requires a
lower regularization we obtain a gain also in terms of the computational cost of the whole restoration procedure.
In all our tests, it is shown that $\|g(u^k)\|_2$ tends to $O(10^{-3})$ or $O(10^{-4})$ at the final FP iterate.

\begin{figure}
    \begin{center}
        \epsfig{figure=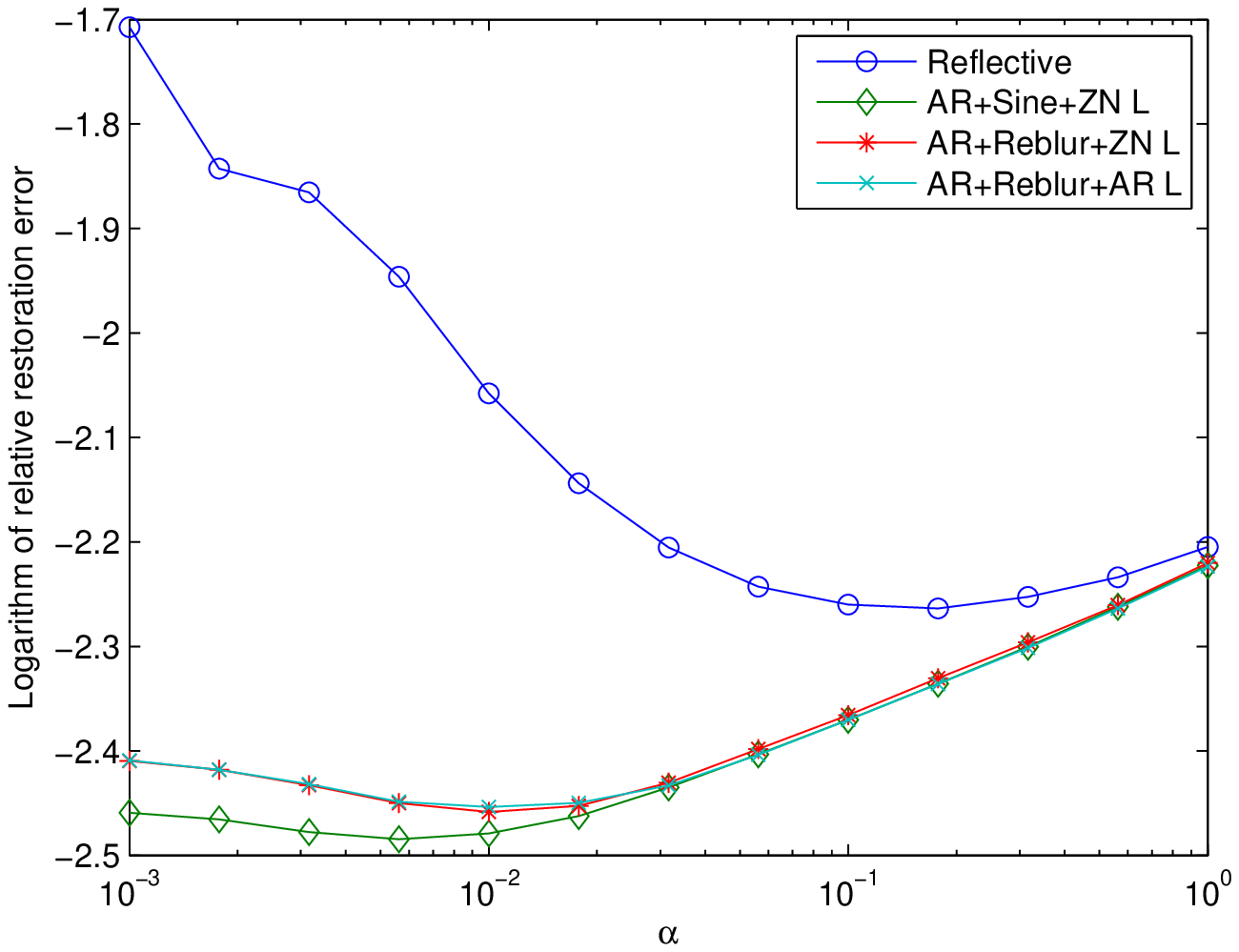,width=9cm}
    \end{center}
\caption{The RRE versus the regularization parameter $\alpha$ for different BCs.} \label{fig:im-tv}
\end{figure}

\begin{figure}
\begin{center}
    \begin{minipage}[c]{6.7cm}
        \centering
        \epsfig{figure=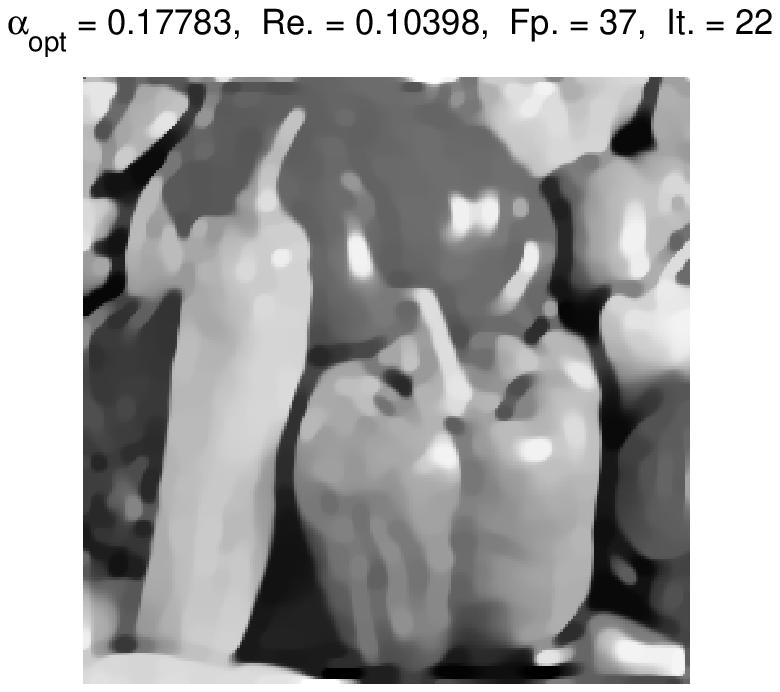,width=6cm}
        \\\vspace{-0.7cm}
        \small{Reflective}
    \end{minipage}\hspace{-1.2cm}
    \begin{minipage}[c]{6cm}
        \centering
        \epsfig{figure=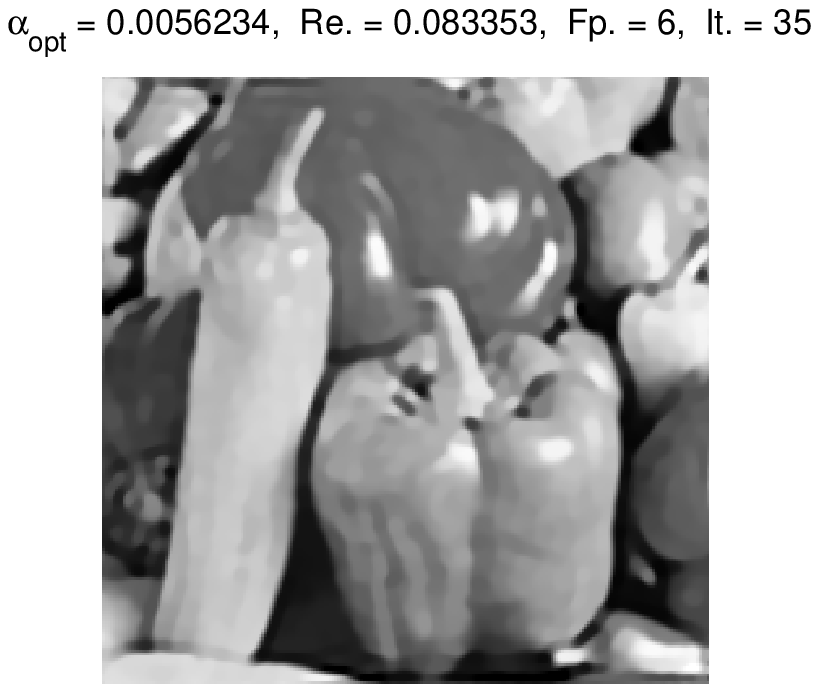,width=6cm}
        \\\vspace{-0.7cm}
        \small{AR+Sine+ZN $\cl$}
    \end{minipage}
\vspace{0.5cm}
\\ \vspace{0.5cm}
    \begin{minipage}[c]{6.7cm}
        \centering
        \epsfig{figure=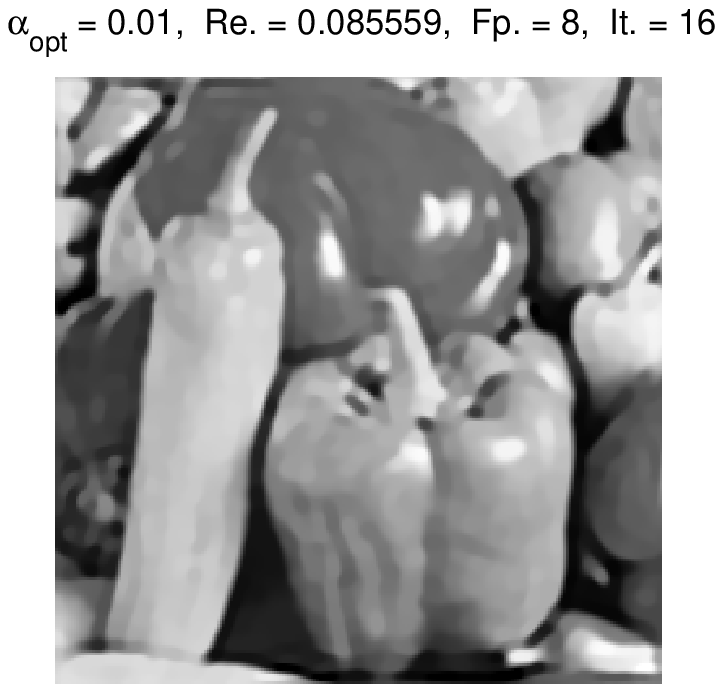,width=6cm}
        \\\vspace{-0.7cm}
        \small{AR+Reblur+ZN $\cl$}
    \end{minipage}\hspace{-1.2cm}
    \begin{minipage}[c]{6cm}
        \centering
        \epsfig{figure=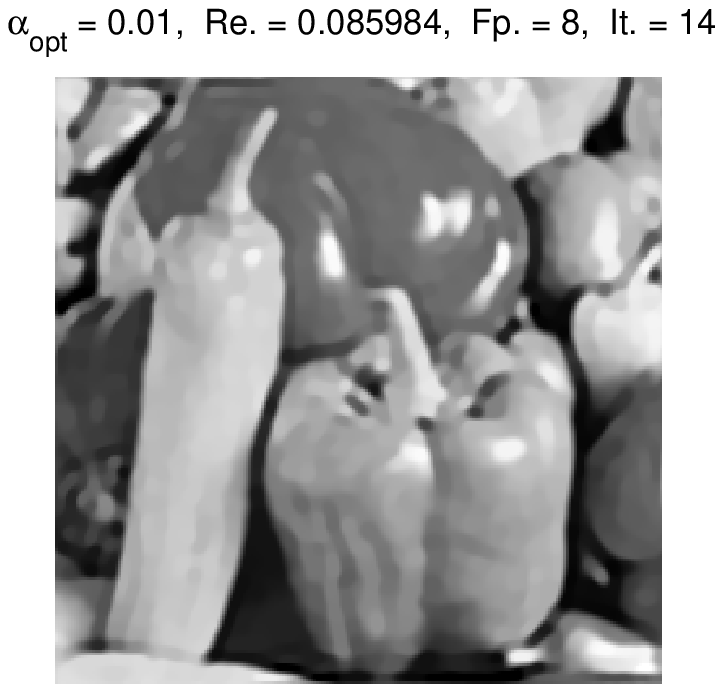,width=6cm}
        \\\vspace{-0.7cm}
        \small{AR+Reblur+AR $\cl$}
    \end{minipage}\hspace{-1.2cm}
\end{center}
\caption{Restored images with different BCs. } \label{fig:res}
\end{figure}

Next, we shall check the quality of restored images by using different BCs. Figure \ref{fig:im-tv} describes
the relative restoration error (RRE) $\|u_\alpha-u\|_2/\|u\|_2$
versus the regularization parameter $\alpha$. Figure  \ref{fig:res} presents
the restored images with optimal value of the parameter $\alpha$.
Like in the 1D case, Figure  \ref{fig:res} shows that the anti-reflective BCs lead
to better restored images and shaper edges with a lower computational cost
than the reflective BCs (note the high reduction in the FP iterations being $\alpha_{\rm opt}$ smaller).

\section{Conclusions}\label{final}

In this paper, we have considered the effect of reflective and anti-reflective BCs, when regularizing
blurred and noisy images via the total variation approach.
In particular, we have studied some preconditioning strategies for the linear systems arising from the
FP iteration given in \cite{VO96}.
In the case of anti-reflective BCs we have also considered a comparison with the
reblurring idea proposed in \cite{DS05}. We recall that the latter has been shown effective
when combined with the Tikhonov regularization and here one of the issues was to verify that reblurring
and total variation can be combined satisfactorily.
Furthermore, the optimal behavior of our preconditioners has been validated numerically.

Beside the computational features of the preconditioning
techniques, we stress the improvement obtained via anti-reflective
BCs both in terms of reconstruction quality and reduction of the computational
cost. In fact, the precision of such BCs was already known in the relevant
literature (see \cite{DES06,ChHa08,perrone} and references there reported).
However, this is the first time that the anti-reflective BCs have been
combined with a sophisticated regularization method, where the use of
fast transforms is very welcome for saving computational cost.
The precision of the antireflective BCs implies also a further reduction
of the computational cost over the other BCs (see numerical results in
Section \ref{num2}) because they require a smaller regularization parameter $\alpha$
and the effectiveness  of the proposed preconditioners increases reducing $\alpha$.

Potential lines of interest for future investigations could include
the use of anti-reflective BCs in the promising split Bregman method
recently proposed  in \cite{GO093} and a more precise clustering analysis of
the preconditioning sequences, in the spirit of Section \ref{sect:spectr}.



\end{document}